\documentclass{article}

\usepackage{amssymb,amsfonts}
\usepackage{gastex}

\usepackage{url}

\def\cqfd{\skip10=\parfillskip\parfillskip=0pt
\enspace\hfill\symbolecqfd\par\parfillskip=\skip10\par\medskip}
\def\symbolecqfd{\rlap{$\sqcap$}$\sqcup$}

\newtheorem{theorem}{Theorem}[section]
\newtheorem{proposition}[theorem]{Proposition}
\newtheorem{lemma}[theorem]{Lemma}
\newtheorem{corollary}[theorem]{Corollary}

\newtheorem{pro-fact}[theorem]{Fact}

\newtheorem{pro-example}[theorem]{Example}
\newenvironment{example}{\begin{pro-example}\rm}{\cqfd\end{pro-example}}

\newtheorem{pro-remark}[theorem]{Remark}
\newenvironment{remark}{\begin{pro-remark}\rm}{\cqfd\end{pro-remark}}

\newcommand{\func}[2]{\textsc{#1}(#2)}
\newcommand{\keyw}[1]{\textbf{#1}}
\newcommand{\com}[1]{\em // #1}

\newenvironment{proof}{\rm \trivlist \item[\hskip \labelsep{\bf
Proof.}]}{\cqfd\endtrivlist}

\def\cqfd{\skip10=\parfillskip\parfillskip=0pt
\enspace\hfill\symbolecqfd\par\parfillskip=\skip10\par\medskip}
\def\symbolecqfd{\rlap{$\sqcap$}$\sqcup$}

\def\proofofTheorem#1{\rm \trivlist \item[\hskip \labelsep{\bf
Proof of Theorem~\ref{#1}.}]}
\def\eopo{\cqfd\endtrivlist}

\def\binom#1#2{{#1\choose#2}}

\def\rank{\textsf{rank}}

\def\inter[#1]{[\![#1]\!]}
\def\inv{^{-1}}

\def \N {\mathbb{N}}
\def \E {\mathbb{E}}
\def \O {\mathcal{O}}
\def \C {\mathcal{C}}
\def \D {\mathcal{D}}
\def \B {\mathcal{B}}
\def \A {\mathcal{A}}

\begin{document}

\markboth{F. Bassino, C. Nicaud, P. Weil}
{Random generation of finitely generated subgroups of 
    a free group}

\title{Random generation of finitely generated subgroups of a free
group\protect\footnote{%
The first and third authors benefitted from the support of the PICASSO
project \textsc{Automata and free groups}.  The third author
acknowledges partial support from the ESF program \textsc{AutoMathA}
and the Algebraic Cryptography
Center, at the Stevens Institute of Technology.}%
}

\author{
Fr\'ed\'erique Bassino,
\small{\url{bassino@univ-mlv.fr}}\protect\footnote{%
Institut Gaspard Monge, Universit\'e de Marne la Vall\'ee,
77454 Marne-la-Vall\'ee Cedex 2, France}
    \\
    \small{Institut Gaspard Monge, Universit\'e de Marne la Vall\'ee}
    \\
    \null
    \\
    Cyril Nicaud,
    \small{\url{nicaud@univ-mlv.fr}}\protect\footnote{%
    Institut Gaspard Monge, Universit\'e de Marne la Vall\'ee,
    77454 Marne-la-Vall\'ee Cedex 2, France}
    \\
    \small{Institut Gaspard Monge, Universit\'e de Marne la Vall\'ee}
    \\
    \null
    \\
    Pascal Weil,
    \small{\url{pascal.weil@labri.fr}}\protect\footnote{%
    LaBRI, 351 cours de la Lib\'eration, 33400 Talence, France.}
    \\
    \small{LaBRI, Universit\'e de Bordeaux, CNRS}
    }

    \date{\today}

\maketitle

\begin{abstract}
    We give an efficient algorithm to randomly generate finitely
    generated subgroups of a given size, in a finite rank free group.
    Here, the size of a subgroup is the number of vertices of its
    representation by a reduced graph such as can be obtained by the
    method of Stallings foldings.  Our algorithm randomly generates a
    subgroup of a given size $n$, according to the uniform
    distribution over size $n$ subgroups.  In the process, we give
    estimates of the number of size $n$ subgroups, of the average rank
    of size $n$ subgroups, and of the proportion of such subgroups
    that have finite index.  Our algorithm has average case complexity
    $\O(n)$ in the RAM model and $\O(n^2\log^2n)$ in the bitcost
    model.
\end{abstract}

\bigskip


\medskip


\vfill

\goodbreak

Algorithmic problems in combinatorial group theory have been the focus
of increased interest (see
\cite{BridsonWise,MSW,KM,jitsukawa,genericKMSS,averageKMSS,MiasUsha}
for recent examples).  This is especially the case for the theory of
free groups and the investigation of their finitely generated
subgroups, which is the focus of this paper.

As in other fields, the investigation of algorithmic problems and of 
their complexity brings to light interesting structural properties of 
the objects under study. One classical approach is to study the 
properties of \textit{random} objects, here of \textit{random} 
finitely generated subgroups of a free group. This naturally depends 
on the probability distribution we adopt, we come back to this below.

The complexity of algorithms is often estimated according to the
worst-case paradigm.  It can also be estimated in average, or
\textit{generically} \cite{genericKMSS}.  Both these concepts also
depend on the choice of a distribution, and can benefit directly from
the enumeration and generation results in this paper.  Random
generation can also be useful to test conjectures or algorithms with
large set of \textit{representative} instances (again, depending on
the choice of a probability distribution), provided that the random
generation algorithm is fast enough.

In this paper, we use the well-known fact (Stallings~\cite{Stallings})
that every finitely generated subgroup of a free group $F$ with basis
$A$ admits a unique graphical representation of the form $(\Gamma,1)$,
where $\Gamma$ is a finite directed graph with $A$-labeled edges and 1
is a designated vertex of $\Gamma$ --- subject to certain
combinatorial conditions, see Section~\ref{sec: representation} for
details.  Then we consider the number of vertices of $\Gamma$ to be a
measure of the \textit{size} of the subgroup represented by
$(\Gamma,1)$.  Note that the size of $H$ is strongly dependent on the
choice of the basis $A$ of $F$.  Very importantly also, for each $n\ge
1$, there are finitely many subgroups of $F$ of size $n$.  The
probability distribution on subgroups discussed in this paper is the
uniform distribution on the set of fixed size subgroups: if there are
$s(n)$ subgroups of size $n$, each has probability $\frac1{s(n)}$.

A large and growing number of algorithmic problems on free groups
admit efficient solutions using these graphical representations (see
\cite{jitsukawa,KM,genericKMSS,averageKMSS,MSW,MVW,RVW} among others),
and this further emphasizes the interest of a random generation scheme
based on this representation.

The main result of the paper is an efficient algorithm to randomly
generate size $n$ subgroups of $F$.  Its average case complexity is
$\O(n)$ in the RAM, or unit-cost model, and $\O(n^2\log^2n)$ in the
bit-cost model, see Section~\ref{sec: complexity} for a more precise
discussion.  Our algorithm actually generates graphical
representations $(\Gamma,1)$ of subgroups, but we want to emphasize
that, as these representations are in bijective correspondence with
finitely generated subgroups of $F$, we truly achieve a uniform
distribution of (size $n$) \textit{subgroups}.

The strategy followed by the algorithm is described in
Section~\ref{sec: enum}.  The algorithm itself is actually simple and
easily implementable, besides being fast.  Its proof is more complex;
it relies on the recursive nature of the combinatorial structures
underlying graphical representations of subgroups (see
Section~\ref{sec:enumeration}, and in particular
Section~\ref{sec:A-graphs}), and we make direct use of the concepts
and the tools of the so-called \textit{recursive method} heralded by
Nijenhuis and Wilf \cite{nw78} and systematized by Flajolet,
Zimmermann and van Cutsem \cite{FZvC}.

In Section \ref{sec:other}, we study the distribution of the ranks of
size $n$ subgroups and we show that if $F$ has rank $r \ge 2$, then
the mean value of the rank of a size $n$ subgroup is $(r-1)n - r\sqrt
n + 1$, with standard deviation $o(\sqrt n)$.  In
Section~\ref{sec:permutations}, we show how our strategy can be
modified (and simplified) to randomly generate in linear time size $n$
finite index subgroups (that is, subgroups of rank $(r-1)n + 1$) ---
even though these groups are negligible among general size $n$
subgroups.  We actually give a precise estimate of the proportion of
such subgroups and we prove that it converges to 0 faster than the
inverse of any polynomial.

The paper closes on a short discussion of related questions, and in
particular of the comparison of our distribution with that which is
induced by the random generation of an $n$-tuple of words and the
consideration of the subgroup they generate, see
\cite{jitsukawa,MTV,MiasUsha} for instance.

\medskip

Throughout this paper, we denote by $|X|$ the cardinality of a set
$X$, and by $\inter[1,n]$ the set $\{1,\ldots,n\}$ (where $n$ is a
positive integer).

%%%%%%%%%%%%%%%%%%%%%%%%
\section{General notions and generation strategy}

%%%%%%%%%%%%%%%%%%%%%%%%
\subsection{Graphical representation of subgroups}\label{sec: representation}

Let $F$ be a free group with finite rank $r \ge 2$ and let $A$ be a
fixed basis of $F$.  We sometimes write $F = F(A)$ and the elements of
$F$ are naturally represented as reduced words over the alphabet $A
\sqcup A\inv$.  It is well-known that the subgroups of $F$ are free as
well.  Moreover, each finitely generated subgroup of $F$ can be
represented uniquely by a finite graph of a particular type, by means
of the technique known as \textit{Stallings foldings} \cite{Stallings}
(see also \cite{Weil,KM,Touikan}).  We refer the reader to the
literature for a description of this very fruitful technique, and we
only record here the results that will be useful for our purpose.

An \textit{$A$-graph} is defined to be a pair $\Gamma = (V,E)$ with $E
\subseteq V\times A\times V$, such that
\begin{itemize}
    \item if $(u,a,v), (u,a,v') \in E$, then $v = v'$;
    \item if $(u,a,v), (u',a,v) \in E$, then $u = u'$.
\end{itemize}
The elements of $V$ are called the \textit{vertices} of $\Gamma$, the
elements of $E$ are its \textit{edges}, and we sometimes write
$V(\Gamma)$ for $V$ and $E(\Gamma)$ for $E$.  We say that $\Gamma$ is
\textit{connected} if the underlying undirected graph is connected.
If $v\in V(\Gamma)$, we say that $v$ is a \textit{leaf} if $v$ occurs
at most once in (the list of triples defining) $E(\Gamma)$ and we say
that $\Gamma$ is \textit{$v$-trim} if no vertex $w\ne v$ is a leaf.
Finally we say that the pair $(\Gamma,v)$ is \textit{admissible} if
$\Gamma$ is a $v$-trim and connected $A$-graph.

Then it is known that:
\begin{itemize}
    \item Stallings foldings associate with each finitely generated
    subgroup $H$ of $F(A)$ a unique admissible pair of the form
    $(\Gamma,1)$, which we call the \textit{graphical representation 
    of $H$} in this paper \cite{Stallings,Weil,KM};
    
    \item every admissible pair $(\Gamma,1)$ is the graphical
    representation of a unique finitely generated subgroup of $F(A)$
    \cite{Stallings,Weil,KM};
    
    \item if $H$ is given by a finite set of generators (in the form 
    of reduced words over $A\sqcup A\inv$) of total length $n$, then 
    the graphical representation of $H$ can be computed in time 
    $\O(n\log^*n)$ \cite{Touikan};
    
    \item if $(\Gamma,1)$ is the graphical representation of $H$, then
    $\rank(H) = |E(\Gamma)| -|V(\Gamma)| + 1$
    \cite{Stallings,Weil,KM};
    
    \item if $(\Gamma,1)$ is the graphical representation of $H$, then
    $H$ has finite index if and only if for each $v\in V(\Gamma)$ and
    for each $a\in A$, there is an edge of the form $(v,a,w)\in
    E(\Gamma)$ \cite{Stallings,Weil,KM}, if and only if $\rank(H) =
    (|A|-1)|V(\Gamma)| + 1$.
\end{itemize}

We sometimes identify $H$ and its graphical representation
$(\Gamma,1)$ --- for instance when we say that we randomly generate
subgroups of $F$: what we generate is actually the graphical
representation of such subgroups.  As explained in the introduction,
we consider the number of vertices to be a measure of the size of
$\Gamma$ and we write $|H| = |\Gamma| = |V|$.  In particular, $F$
has finitely many subgroups of size $n$.

%%%%%%%%%%%%%%%%%%%%%%%%
\subsection{Enumeration and random generation}\label{sec: enum}

As we shall see, $A$-graphs fall in the category of decomposable
structures, that is, structures that can be built from unit elements
and from operations such as the union, direct product, set formation,
etc.  We will use the so-called \textit{recursive method} to enumerate
and to randomly generate such structures \cite{FZvC}.  Details are
given further in the paper, concerning the enumeration
(Section~\ref{sec:enumeration}) and the random generation algorithm
and its complexity (Section~\ref{sec:randomgeneration}).  At this
point, let us simply say that the random generation of size $n$
$A$-graphs requires a pre-computation phase in $\O(n)$, after which
each draw takes time $\O(n)$.

The rest of this section is devoted to an overview of our strategy.

\begin{remark}\label{rem:boltzamnn1}
    There exists another method than the recursive method, to derive a
    random generation algorithm from a combinatorial specification,
    this time according to a \textit{Boltzmann distribution}.  Recall
    that, in such a distribution, an object $\gamma$ receives a
    probability essentially proportional to an exponential of its size
    $|\gamma|$.  (More precisely this probability depends upon a
    positive real parameter $x$, and it is proportional to
    $x^{|\gamma|}$ when $\gamma$ is an unlabeled structure and to
    $x^{|\gamma|}/|\gamma|!$ when $\gamma$ is labeled; see
    Section~\ref{sec labeled vs unlabeled} below about labeled vs.
    unlabeled structures.)  In particular, Boltzmann samplers do not
    generate objects of a fixed size.  They depend on the real
    parameter $x>0$ and, for any given integer $n$, the value of $x$
    can be chosen such that the average size of the generated elements
    is $n$.  Even though the size of the objects generated is not
    fixed, Boltzmann samplers guarantee that two elements of the same
    size have the same probability to be generated.
    
    A method to systematically produce Boltzmann samplers was recently
    introduced by Duchon, Flajolet, Louchard and Schaeffer
    \cite{dfls04} for labeled structures (Flajolet, Fusy and Pivoteau
    for unlabeled structures in \cite{FFP07}).  The evaluation of $x$
    is the only required precomputation and the complexity of
    generation itself is linear as long as small variations in size
    are allowed.  This approach can also be used for exact-size
    generation, but in the case of $A$-graphs it is less efficient
    than the recursive method (see Remark \ref{rem:boltzmann2}).
\end{remark}

%%%%%%%%%%%%%%%%%%%%%%%%
\subsubsection{We count labeled $A$-graphs}\label{sec labeled vs unlabeled}

Enumeration for us, is the enumeration of structures up to
isomorphism.  The structures which we want to generate are admissible
pairs $(\Gamma,1)$, that is, $A$-graphs with one vertex labeled 1,
that are connected and 1-trim.  We later use the phrase
\textit{admissible $A$-graphs}.  Leaving aside for a moment the
properties of connectedness and 1-trimness, we are interested in
$A$-graphs with a distinguished vertex.  This is an intermediary
situation between labeled and unlabeled structures, which are two
great categories of structures for which there exist a large toolkit
for enumeration and random generation \cite{FZvC,FSBook,dfls04,FFP07}.

An $A$-graph $\Gamma = (V,E)$ of size $n$ is said to be
\textit{labeled} if it is equipped with a bijection $\lambda\colon
\inter[1,n] \rightarrow V$.  Of course, there are $n!$ different
such bijections, but some of them may yield isomorphic labeled
$A$-graphs.  For instance, if $E = \emptyset$ (so that $\Gamma$
consists of $n$ isolated vertices), all labelings of $\Gamma$ are
isomorphic.

\begin{figure}[htbp]
\begin{picture}(114,39)(0,-39)
\gasset{Nw=4,Nh=4}

\node(n0)(8.0,-4.0){}

\node(n1)(8.0,-28.0){}

\put(3,-37){$\Gamma_1$}

\node(n2)(30,-4.0){}

\node(n3)(30,-28.0){}

\put(25,-37){$\Gamma_2$}

\node(n4)(52,-4.0){}

\node(n5)(72,-4.0){}

\node(n6)(52,-28.0){}

\node(n7)(72,-28.0){}

\put(47,-37){$\Gamma_3$}

\node(n8)(94,-4.0){}

\node(n9)(114,-4.0){}

\node(n10)(94,-28.0){}

\node(n11)(114,-28.0){}

\put(89,-37){$\Gamma_4$}

\drawedge(n4,n5){$a$}

\drawedge(n5,n7){$b$}

\drawedge(n7,n6){$a$}

\drawedge(n6,n4){$b$}

\drawedge(n8,n9){$a$}

\drawedge(n9,n11){$a$}

\drawedge(n11,n10){$b$}

\drawedge(n10,n8){$b$}

\drawedge[curvedepth=4.0](n2,n3){$a$}

\drawedge[curvedepth=4.0](n3,n2){$b$}

\drawedge[curvedepth=4.0](n0,n1){$a$}

\drawedge[curvedepth=4.0](n1,n0){$a$}

\end{picture}
\caption{Four $A$-graphs with different numbers of non-isomorphic 
labelings}\label{Figure labelings}
\end{figure}
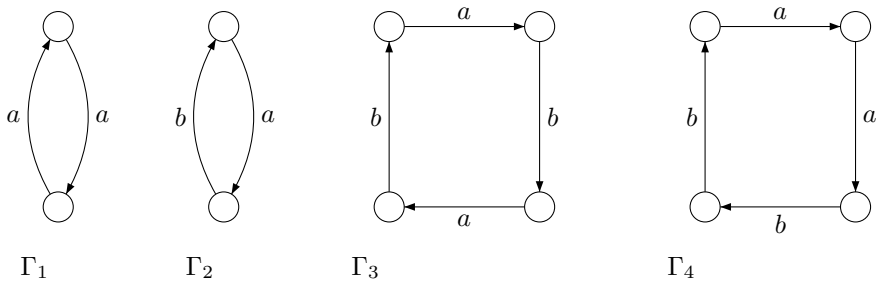

\begin{example}\label{ex cycles}
    Consider the $A$-graphs in Figure~\ref{Figure labelings}. Then 
    all labelings of $\Gamma_1$ are isomorphic but $\Gamma_2$ has 2 
    non-isomorphic labelings. How many non-isomorphic labelings does 
    $\Gamma_3$ have? And $\Gamma_4$?
\end{example}

Counting structures such as $A$-graphs can lead to complex
considerations involving, say, the automorphism groups of each of the
connected components of the $A$-graphs, but standard methods of
analytic combinatorics actually solve this enumeration problem for us,
see Section~\ref{sec:EGS}.  These methods rely on the use of labeled
structures, and we will therefore enumerate and generate labeled
$A$-graphs.  Why this is justified even for the purpose of randomly
generating admissible pairs is discussed in
Section~\ref{sec:rejection}.

%%%%%%%%%%%%%%%%%%%%%%%%
\subsubsection{Generating series}\label{sec:EGS}

Let $\A$ be a class of combinatorial structures.  If $\A$ has
$a_n$ elements of size $n$, the formal power series $\sum_n a_n z^n$
and $\sum_n \frac{a_n}{n!} z^n$ are called respectively the
\textit{(ordinary) generating series} and the \textit{exponential
generating series} (\textit{EGS}) of the class $\A$.

As it turns out, certain operations on classes of labeled
combinatorial structures have a direct translation over their EGS. For
instance, suppose that $\A$ is the union of the disjoint classes
$\B$ and $\C$ (that is, a size $n$ element of $\A$ is a size
$n$ element of either $\B$ or $\C$), and let $A(z)$, $B(z)$ and
$C(z)$ be the EGS of the three classes.  It is immediate that $A(z) =
B(z) + C(z)$.

For more complex operations, one needs to handle the question of
relabelings.  If $S$ is a size $n$ structure with a labeling function
$\lambda$, we say that $\mu$ is an expansion of $\lambda$ if the
domain of $\mu$ is of the form $\{k_1,\ldots,k_n\} \subset \N$ with
$k_1 < \cdots < k_n$ and $\mu(k_i) = \lambda(i)$ for each $i$.  If
$S_1,\ldots,S_r$ are structures of size $n_1, \ldots, n_r$ with
labeling functions $\lambda_1, \ldots, \lambda_r$, then the sequence
$S = (S_1, \ldots, S_r)$ is a structure of size $n = \sum_i n_i$ and
we say that a labeling $\lambda$ of $S$ is \textit{compatible} with
the $\lambda_i$ if it is obtained by the combination of expansions of
the $\lambda_i$, whose domains form a partition of $\inter[1,n]$.
In particular, $\lambda$ is not uniquely determined by the
$\lambda_i$.

\begin{figure}[htbp]
\begin{picture}(114,39)(0,-39)
\gasset{Nw=4,Nh=4}

\node(n0)(8.0,-4.0){1}

\node(n1)(8.0,-28.0){2}

\put(3,-37){$\Gamma_1$}

\node(n2)(30,-4.0){1}

\node(n31)(22,-28.0){3}

\node(n32)(38,-28.0){2}

\put(18,-37){$\Gamma_2$}

\node(n10)(55.0,-4.0){1}

\node(n11)(55.0,-28.0){3}

\node(n12)(71,-4.0){2}

\node(n131)(63,-28.0){5}

\node(n132)(79,-28.0){4}

\put(50,-37){a labeling of $(\Gamma_1,\Gamma_2)$\dots}

\node(n20)(96.0,-4.0){2}

\node(n21)(96.0,-28.0){5}

\node(n22)(112,-4.0){1}

\node(n231)(104,-28.0){4}

\node(n232)(120,-28.0){3}

\put(93,-37){and another}

\drawedge(n2,n32){$a$}

\drawedge(n32,n31){$a$}

\drawedge(n31,n2){$b$}

\drawedge[curvedepth=4.0](n0,n1){$b$}

\drawedge[curvedepth=4.0](n1,n0){$a$}

\drawedge(n12,n132){$a$}

\drawedge(n132,n131){$a$}

\drawedge(n131,n12){$b$}

\drawedge[curvedepth=4.0](n10,n11){$b$}

\drawedge[curvedepth=4.0](n11,n10){$a$}

\drawedge(n22,n232){$a$}

\drawedge(n232,n231){$a$}

\drawedge(n231,n22){$b$}

\drawedge[curvedepth=4.0](n20,n21){$b$}

\drawedge[curvedepth=4.0](n21,n20){$a$}

\end{picture}
\caption{Two labeled $A$-graphs and two compatible labelings of the
sequence they compose}\label{Figure relabeling}
\end{figure}
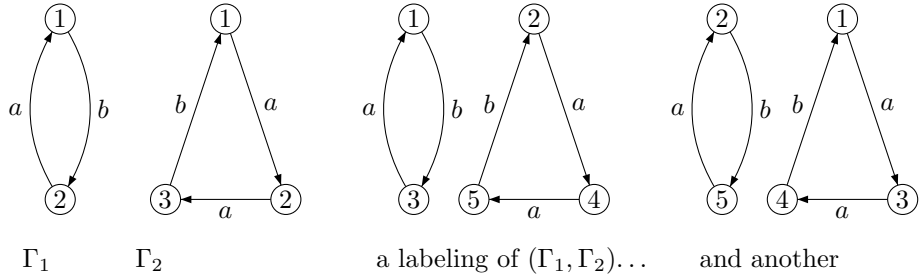

\begin{example}
    Figure~\ref{Figure relabeling} shows two labeled $A$-graphs of
    sizes 2 and 3, and several compatible labelings of the sequence
    they compose.
\end{example}    

We record here three such operations, that will be important for our
purpose, and we refer the readers to \cite{FSBook1,FSBook} for the
proof of this important result.

\begin{proposition}\label{EGS calculus}
    Let $\A$ be a class of labeled structures with EGS $A(z)$.
    \begin{itemize}
        \item Let $\B$ be the class of \textit{sequences of
        structures from class $\A$} -- that is, a size $n$ labeled
        structure in $\B$ is a tuple $(S_1,\ldots,S_r)$ of labeled
        structures in $\A$ (each $S_i$ of size $n_i$ with $n =
        \sum_i n_i$), equipped with a labeling function compatible
        with the labelings of the $S_i$.  Then the EGS of $\B$ is
        $$A(z) = \frac{1}{1-B(z)}.$$
        
        \item Let $\C$ be the class of \textit{cycles of structures from
        class $\A$}, where a cycle of structures is an equivalence
        class of sequences of structures in $\A$, under the
        relation which identifies a sequence with its cyclic
        permutations.  Then the EGS of $\C$ is
        $$C(z) = \log\left(\frac{1}{1-B(z)}\right).$$

        \item Let $\D$ be the class of \textit{sets of structures
        from class $\A$}, where a set of structures is an
        equivalence class of sequences of structures in $\A$, under
        the relation which identifies a sequence with its
        permutations.  Then the EGS of $\D$ is
        $$D(z) = \exp\big(B(z)\big).$$
    \end{itemize}
\end{proposition}

Let us first apply this \textit{calculus} to a simple example.  Let
$\A$ be the class consisting of a single graph, with one vertex and
no edges (the vertex being necessarily labeled 1).  Its EGS is $A(z) =
z$.  The class of labeled sequences of structures in $\A$ is in bijection
with the class of labeled line graphs of the form
\begin{figure}[htbp]
\begin{center}\compatiblegastexun
\unitlength=1mm
\begin{picture}(65,3)(0,-3)
 \thicklines
 %\thinlines
  \setvertexdiam{5}
  \letvertex A=(0,-2)    \drawcircledvertex(A){$i_1$}
  \letvertex B=(15,-2)    \drawcircledvertex(B){$i_2$}
  \letvertex C=(30,-2)   \drawcircledvertex(C){$i_3$}
  \letvertex D=(65,-2)    \drawcircledvertex(D){$i_n$}
  \letvertex E=(45,-2)

  \drawedge(A,B){}
  \drawedge(B,C){}
  \drawedge(C,E){}
\put(50.5,-3.0){$\cdots$}
\end{picture}
\end{center}
\end{figure}
with $\{i_1,\ldots,i_n\} = \inter[1,n]$, and its EGS is, according
to Proposition~\ref{EGS calculus}, equal to $\frac{1}{1-z} = \sum_n
z^n$.  This corresponds to the fact that the number of such size $n$
sequences is $n!$ (so that its quotient by $n!$ is 1).

Similarly, the EGS of the class of cycles of structures in $\A$,
that is, the class of labeled cyclic graphs (such as the graphs in
Example~\ref{ex cycles} that are labeled with a single letter) is
$\log\left(\frac{1}{1-z}\right)$.

We extend these examples in Section~\ref{sec:A-graphs} to compute
the EGS of labeled $A$-graphs.

%%%%%%%%%%%%%%%%%%%%%%%%
\subsubsection{A rejection algorithm}\label{sec:rejection}

The enumeration and random generation of labeled $A$-graphs is however
not our objective.  We want to generate admissible pairs, that is,
$A$-graphs with a distinguished vertex 1, that are connected and
1-trim.  This will be achieved by a \textit{rejection algorithm} (see
\cite{dev86}).

Suppose one wants to draw a number between 1 and 5, using a dice. It 
is natural to throw the dice repeatedly, until the result is 
different from 6. This is a semi-algorithm, since in the worst case 
it may never end --- if we draw only 6's ---, but we will loosely call 
it an algorithm.

Formally, suppose we want to generate elements of a set $X$, according
to a probability law $p_X$.  Suppose that $X$ is a subset of $Y$, and
that we have a probability law $p_Y$ on $Y$, whose restriction to $X$
is $p_X$.  If we have an algorithm to generate elements of $Y$
according to $p_Y$, we may use this algorithm to generate elements of
$X$ as follows: repeatedly draw an element of $Y$, reject it if it is
not in $X$, stop if it is in $X$.  The average complexity of such an
algorithm depends on the complexity of the generating algorithm on
$Y$, on the complexity of the test whether an element of $Y$ is in
$X$, and on the average number of rejects.  One can show that if
$p_Y(X)$ is the probability for an element of $Y$ to be in $X$, the 
average number of rejects is $1/p_Y(X)$.

Concretely, we will show that the probability $p_n$ for a size $n$
labeled $A$-graph to be connected and $1$-trim (rigorously:
$\lambda(1)$-trim) tends to 1 when $n$ tends to infinity
(Section~\ref{sec:connected} and~\ref{sec:trim}).  This justifies the
use of a rejection algorithm since the average number of rejects tends
to $0$ when $n$ tends to infinity.

There remains one problem: such a rejection algorithm will generate 
labeled connected 1-trim $A$-graphs, and we are interested only in 
the information contained in $(\Gamma,\lambda(1))$ -- that is, we do 
not care how the $n-1$ vertices different from $\lambda(1)$ are 
labeled. In other words, an admissible pair, obtained from a labeled 
$A$-graph by forgetting the labeling of the vertices numbered 2 to 
$n$, will be counted several times. The following lemma shows that 
this is not an obstacle.

% Let $v$ be a vertex of $\Gamma$.  We say that the pair $(\Gamma,v)$ is
% labeled if it is equipped with a bijection $\lambda\colon
% \{1,\ldots,n\} \rightarrow V$ such that $\lambda(1) = v$.  Again, there
% are $(n-1)!$ such labelings. We claim the following.

\begin{lemma}\label{fact labeled}
    Let $\Gamma$ be an $A$-graph of size $n$ and let $v$ be a vertex
    of $\Gamma$.  If $\Gamma$ is connected, then there are $(n-1)!$
    isomorphism classes of labeled structures on $\Gamma$ such that
    vertex $v$ is labeled 1.
\end{lemma}

\begin{proof}
Let $\bar A = \{\bar a \mid a \in A\}$ be a disjoint copy of $A$, and
let $\tilde\Gamma$ be obtained from $\Gamma$ by adding, for each edge
$(x,a,y)$, a new edge $(y,\bar a,x)$.  It is immediately verified that
$\tilde\Gamma$ is an $(A\cup\bar A)$-graph.  For each vertex $w\ne v$,
let $u_w$ be a finite word on the alphabet $A \cup \bar A$ labeling a
path in $\tilde\Gamma$ from $v$ to $w$.  Then $w$ is the unique vertex
accessible from $v$ following a path labeled $u_w$, and the words
$u_w$ are pairwise distinct.  This observation guarantees that
distinct labelings of $\Gamma$ mapping 1 to $v$, are non isomorphic.
\end{proof}

Thus each size $n$ admissible pair $(\Gamma,1)$ is counted the same
number of times, namely $(n-1)!$ times.  Therefore, applying a
rejection algorithm that randomly generates size $n$ labeled
connected 1-trim $A$-graphs and forgetting labels $2$ to $n$ also
guarantees a random generation of admissible pairs for the uniform
distribution on all admissible pairs of size $n$.

Summarizing the algorithmic strategy, we will randomly and equally
likely generate a labeled $A$-graph, reject it if it is not connected
and 1-trim, draw another labeled $A$-graph, etc, until we draw a
connected 1-trim labeled $A$-graph.  We then ignore the labeling of
the vertices numbered 2 to $n$.  Details on the algorithm and its
complexity are discussed in Section~\ref{sec:randomgeneration}.

For convenience, we will call a labeled $A$-graph $\Gamma$
\textit{admissible} if the pair $(\Gamma,1)$ is admissible in the
sense of Section~\ref{sec: representation}.

%%%%%%%%%%%%%%%%%%%%%%%%
\section{Enumeration of $A$-graphs}\label{sec:enumeration}

We first observe that in an $A$-graph $\Gamma = (V,E)$, for each $a\in
A$, the edges in $E$ of the form $(x,a,y)$ can be interpreted as the
description of a partial injection from $V$ to $V$ (partial means that
the domain of this injection is a subset of $V$).  If the $A$-graph is
labeled and has size $n$, each letter can therefore be interpreted as
a partial injection from $\inter[1,n]$ to itself.  The labeled
$A$-graph itself can then be seen as an $A$-tuple of partial
injections.

In this section, we discuss the enumeration of partial injections on
$\inter[1,n]$, the probability for an $A$-tuple of such partial
injections to yield a labeled connected $A$-graph, the probability for
that graph to be 1-trim, and finally the number of size $n$ subgroups.

%%%%%%%%%%%%%%%%%%%%%%%%
\subsection{Partial injections and $A$-graphs}\label{sec:A-graphs}

Each partial injection is a set of disjoint cycles and non-empty
sequences (in analogy to the decomposition of a permutation as a union
of cycles).  The EGS for cycles is $\log(\frac{1}{1-z})$, and that for
sequences is $\frac{1}{1-z}$ (Section~\ref{sec:EGS}).  It follows that
the EGS for non-empty sequences is $\frac{1}{1-z} -1 = \frac{z}{1-z}$,
and the EGS for the union of the (disjoint) classes of cycles and
non-empty sequences is $\frac{z}{1-z} + \log(\frac{1}{1-z})$.  Then
Proposition~\ref{EGS calculus} shows that the EGS for partial
injections is
$$I(z) = \exp \left( \frac{z}{1-z} + \log \Big(\frac{1}{1-z}\Big)\right) =
\frac{1}{1-z}\,\exp\left(\frac{z}{1-z}\right).$$
Let $I_n$ be the number of partial injections from $\inter[1,n]$ to
itself, so that $I(z) = \sum_n \frac{I_n}{n!}z^n$.

\begin{remark}
    The series $I(z)$ turns out to be also the (ordinary) generating 
    series of the average number of increasing subsequences in a 
    random permutation.
    The first values of the sequence $(I_n)_{n\geq 0}$, referenced
    \textit{EIS} \textbf{A002720} in~\cite{sloane}  are
    $$1,\ 2,\ 7,\ 34,\ 209,\ 1546,\ 13327,\ 130922,\ 1441729,\
    17572114,\ 234662231, \ldots$$
\end{remark}    

The above expression of $I(z)$ yields a simple recurrence relation for
the sequence $(I_n)_n$.  Indeed, we find that the series
$I'(z)=\sum_{n\geq 0}\frac{I_{n+1}}{n!}z^n$ is equal to
$$I'(z)=\frac{2-z}{(1-z)^3} \,\exp\left(\frac{z}{1-z}\right) =\frac{2-z}{(1-z)^2}
\, I(z).$$
Thus $(1-z)^2 I'(z) = (2-z) I(z)$ and the following recurrence
relation follows:
\begin{equation} \label{eq:rec}
\forall n \geq 2, \quad I_n = 2 n \, I_{n-1}-(n-1)^2I_{n-2},
\end{equation}
with $I_0 = 1$ and $I_1 = 2$.

\begin{lemma} \label{minomajo}
    For each integer $n \geq 1$, we have
    $$(n+1)!  \enspace\le\enspace (n+1) I_{n-1} \enspace\le\enspace
    I_n \enspace\le\enspace n\, e^{1/\sqrt{n}} I_{n-1}
    \enspace\le\enspace n!\ e^{2\sqrt{n}-1}.$$
\end{lemma}

\begin{proof}
We proceed by induction on $n$, noting that $I_0 = 1$.  The inequality
$(n+1)I_{n-1} \le I_n$ is verified for $n = 1$ and $n = 2$.  Suppose
that $n\ge 2$ and $(n+1)I_{n-1} \le I_n$.  By the recurrence relation,
we have
\begin{eqnarray*}
    I_{n+1} \enspace=\enspace (2n + 2) I_n - n^2 I_{n-1} &=& (2n + 2) 
    I_n - (n + 1)^2 I_{n-1} + (2n + 1)I_{n-1} \\
    &\ge& (2n+2) I_n - (n+1)I_n + (2n+1) I_{n-1} \\
    &\ge& (n+1) I_n + 2n I_{n-1} \\
    &\ge& (n+1) I_n + I_n \enspace=\enspace (n+2)I_n,
\end{eqnarray*}
with the last inequality derived from the recurrence relation on the
$I_n$.  Thus, for each $n\ge 1$, we have $I_n \ge (n+1)I_{n-1}$.  The
inequality $(n+1)!  \le (n+1)I_{n-1}\le I_n$ follows immediately.

For $n \geq 1$ let $u_n=\frac{I_n}{I_{n-1}}$.  We proceed by induction
on $n$ to prove that $u_n \leq n e^{1/\sqrt{n}}$, noting that $u_1 = 2
\leq e$.  From Equation~(\ref{eq:rec})
$$u_{n+1} = 2(n+1)-\frac{n^2}{u_n} \leq 2(n+1) -n e^{-1/\sqrt{n}}$$
It remains to show that 
$$2(n+1) -n e^{-1/\sqrt{n}} \leq (n+1) e^{1/\sqrt{n+1}}.$$
or
$$2(n+1) \leq n e^{-1/\sqrt{n}} + (n+1) e^{1/\sqrt{n+1}}.$$
For any real number $x$,
$$e^x \geq 1+x+\frac{x^2}{2}+\frac{x^3}{6}.$$ 
Therefore
$$n e^{-1/\sqrt{n}} + (n+1) e^{1/\sqrt{n+1}} \geq 2(n+1) +
\sqrt{n+1}-\sqrt{n} +\frac{1}{6\sqrt{n+1}} - \frac{1}{6\sqrt{n}}$$
and we want to show that
$$\sqrt{n+1} - \sqrt{n} + \frac{1}{6\sqrt{n+1}} - \frac{1}{6\sqrt{n}}
= \sqrt{n+1}-\sqrt{n} +
\frac{\sqrt{n}-\sqrt{n+1}}{6\sqrt{n}\sqrt{n+1}} \ge 0,$$
or equivalently,
\begin{eqnarray*}
    6(n+1)\sqrt n - 6n\sqrt{n+1} + \sqrt n - \sqrt{n+1} &\ge& 0 \\
    (6n+7)\sqrt n &\ge& (6n+1)\sqrt{n+1}\\
    (6n+7)^2 n &\ge& (6n+1)^2(n+1).
\end{eqnarray*}
Now the difference $\left((6n+7)^2 n\right) - 
\left((6n+1)^2(n+1)\right)$ is equal to $36n^2+36n-1$, which is 
positive for all $n\ge 1$. This completes the proof that $u_n \leq n 
e^{1/\sqrt{n}}$ for all $n\ge 1$.

Consequently $I_n \enspace\le\enspace n\, e^{1/\sqrt{n}} I_{n-1}$ and
$$I_{n}=\frac{I_n}{I_0}=\prod_{i=1}^nu_i\leq n! \,
e^{\sum_{i=1}^n\frac{1}{\sqrt{i}}}.$$
As the function $x \mapsto \frac{1}{\sqrt{x}}$ is decreasing on the
positive domain, we find that $\frac{1}{\sqrt{i+1}} \leq \int_i^{i+1}
\frac{dx}{\sqrt{x}}$ for each $i \geq 1$, and
$$\sum_{i=2}^n\frac{1}{\sqrt{i}} =
\sum_{i=1}^{n-1}\frac{1}{\sqrt{i+1}}
\le  \int_1^n \frac{dx}{\sqrt{x}}=2\sqrt{n}-2.$$ 
Thus $e^{\sum_{i=1}^n\frac{1}{\sqrt{i}}}\leq e^{2\sqrt{n}-1}$, which
concludes the proof.
\end{proof}

\begin{remark}\label{rem:boltzmann2}
    The computation of $I(z)$ allows us to justify our assertion that
    Boltzmann samplers are less efficient than the random generation
    based on the recursive method, see Remark~\ref{rem:boltzamnn1}.
    More precisely, the behavior of Boltzmann samplers is often such
    that the size of the generated object sits between
    $(1-\varepsilon)n$ and $(1+\varepsilon)n$ with high probability.
    In our case, it is essential that the tuple of partial injections
    we generate all have the same size.  It is often the case that a
    Boltzmann sampler can be used to produce an exact-size sampler,
    using a rejection algorithm.  In the case of partial injections
    however, as the distribution of the sizes of partial injections is
    not sufficiently concentrated around the mean size, each draw of a
    random partial injection of size exactly $n$ takes times
    $\O(n^{7/4})$, which is not very satisfactory.  Here is why.
    
    The mean size of a partial injection produced under the
    exponential Boltzmann model is (see \cite[Proposition 1]{dfls04}):
    $$\mathbb{E}_x(\textrm{size of a partial injection}) = x
    \frac{I'(x)}{I(x)} = x\frac{2-x}{(1-x)^2}$$
    and its variance is
    $$\sigma^2_x(\textrm{size of a partial injection}) = \frac{d}{dx}
    \mathbb{E}_x(\textrm{size of a partial injection}) =
    \frac{2}{(1-x)^3}.$$
    Thus, to generate partial injections of expected size
    $\mathbb{E}_x=n$, one has to choose $x=1-1/\sqrt{n+1}$.  In this
    case , $\sigma^2_x=2(n+1)^{3/2}$.  From \cite[Theorem 4]{dfls04}
    dealing with H-admissible generating functions (see Section
    \ref{sec:saddlepoint}) the exact-size generation requires $\sqrt{2
    \pi} \sigma_x=\O(n^{3/4})$ rejections in average and the overall
    cost of exact-size sampling is $\O(n \sigma_x)=\O(n^{7/4})$ in
    average.
\end{remark}
As discussed at the beginning of this section, if $r = |A|$, then
a labeled $A$-graph of size $n$ can be assimilated to a $r$-tuple of
partial injections on $\inter[1,n]$, so the EGS of labeled $A$-graphs
is
$$\sum_{n \geq 0}\frac{I^r_n}{n!} \ z^n.$$

%%%%%%%%%%%%%%%%%%%%%%%%
\subsection{Connectedness}\label{sec:connected}

Recall that an $A$-graph is connected if the underlying
undirected graph is connected.  In this section, we show the following
result.

\begin{theorem}\label{asympt connected}
    Let $A$ be an alphabet of cardinality $r\geq 2$ and let $p_n$ be
    the probability for an $n$-vertex labeled $A$-graph to be
    connected.  Then $\lim_{n\mapsto\infty}p_n = 1$ and more 
    precisely, $p_n = 1 -\frac{2^r}{n^{r-1}} +o(\frac{1}{n^{r-1}})$.
\end{theorem}

In particular, this shows that labeled $A$-graphs are asymptotically
connected.  The proof of Theorem~\ref{asympt connected}, given below,
relies on the following theorem, due to Bender (see \cite[p.
497]{ben74} for a survey and \cite{ben7475} for a complete proof).

\begin{theorem} \label{thbender}
    Let $F(z,y)$ be a two-variable real function which is analytic at
    $(0,0)$.  Let $J(z)=\sum_{n>0}j_nz^n$, $C(z)=\sum_{n>0}c_nz^n$ and
    $D(z)=\sum_{n>0}d_nz^n$ be functions such that
    $$ C(z)=F(z, J(z)) \quad \mbox{and} \quad D(z)=\frac{\partial
    F}{\partial y}(z, J(z)).$$
    If the sequence $(j_n)_{n>0}$ satisfies $j_{n-1}=o(j_n)$, and if
    for some $s \geq 1$, we have $\sum_{k=s}^{n-s}|j_kj_{n-k}| =
    \O(j_{n-s})$, then
    $$c_n = \sum_{k=0}^{s-1}d_kj_{n-k} + \O\left(j_{n-s}\right).$$
\end{theorem}

\proofofTheorem{asympt connected}
Let $J(z) = \sum_{n \geq 1}j_n z^n$ where $j_n = \frac{I_n^r}{n!}$.
Then the EGS of labeled $A$-graphs is $1 + J(z)$.

Decomposing these graphs into their connected components (connected
components of the underlying undirected graph) and using
Proposition~\ref{EGS calculus}, we find that $1 + J(z) = \exp(C(z))$,
where $C(z) = \sum_{k\geq 1}c_k z^k$, with $c_k=\frac{C_k}{k!}$ and
$C_k$ is the number of connected labeled $A$-graphs with $k$ vertices.

It follows that $C(z)=\log(1+J(z))$.  We note that the map $F(z,y) =
\log(1+y)$ is analytic at $(0,0)$ and we let
$$D(z)=\frac{\partial F}{\partial y}(z, J(z)) =\frac{1}{1+J(z)}.$$

By Lemma \ref{minomajo}, $I_{n-1} \leq \frac{I_n}{n+1}$, and therefore
$$j_{n-1} = \frac{I_{n-1}^r}{(n-1)!} \leq
\frac{I_n^r}{(n+1)^r(n-1)!} \le \frac{I_n^r}{n^{r-1}\,n!} =
\frac{j_n}{n^{r-1}}.$$
In particular, we have
\begin{equation}\label{eqj}
    j_{n-1} = \O\left(\frac{j_n}{n^{r-1}} \right) = o\left(j_n\right).
\end{equation}

We now want to verify whether $\sum_{k=s}^{n-s} j_k j_{n-k} =
\O (j_{n-s})$ (this is the last hypothesis of Theorem~\ref{thbender},
for a fixed $s \geq 1$).  Let $S$ be the sum above.  By symmetry, we get
$$S = \sum_{k=s}^{n-s} \frac{I_k^r}{k!}\, \frac{I_{n-k}^r}{(n-k)!}
\leq 2 \left(\sum_{k=s}^{\lfloor
n/2\rfloor} \frac{I_k^r}{k!}\, \frac{I_{n-k}^r}{(n-k)!}\right).$$
We show that for $n$ large enough, 
the finite sequence $\left(\frac{I_k^r}{k!}\,
\frac{I_{n-k}^r}{(n-k)!}\right)_{s \le k \le \lfloor n/6\rfloor}$ is
decreasing.  From Equation~(\ref{eq:rec}), we have
$$\frac{I_{k+1}}{I_k} \enspace=\enspace
2(k+1)-k^2\frac{I_{k-1}}{I_{k}} \enspace\leq\enspace 2(k+1);$$
and by Lemma \ref{minomajo}, we have $\frac{I_{n-(k+1)}}{I_{n-k}} \le
\frac{1}{n-k+1}$. Therefore
$$\frac{I_{k+1}^rI_{n-(k+1)}^r}{I_k^rI_{n-k}^r} \le
\frac{2^r(k+1)^r}{(n-k+1)^r}.$$
Moreover $\frac{k!}{(k+1)!}
\,\frac{(n-k)!}{(n-k-1)!}=\frac{n-k}{k+1}$, and it follows that
\begin{eqnarray*}
    \frac{I_{k+1}^rI_{n-(k+1)}^r}{(k+1)!(n-k-1)!}\
    \frac{k!(n-k)!}{I_k^rI_{n-k}^r} &\le&
    \frac{2^r(k+1)^r}{(n-k+1)^r}\,\frac{n-k}{k+1}\\
    &\le& \frac{2^r(k+1)^{r-1}}{(n-k+1)^{r-1}} \, \frac{n-k}{n-k+1}\\
    &\le& \frac{2^r(k+1)^{r-1}}{(n-k+1)^{r-1}}.
\end{eqnarray*}    
This value is less than or equal to $1$ when $\frac{k+1}{n-k+1}$ is
less than or equal to $c = 2^{-\frac{r}{r-1}}$. Since $r\geq 2$,
we have $\frac14\leq c\leq\frac12$,  and for any $k\leq
\frac{n-3}5$, $\frac{k+1}{n-k+1}\leq\frac14\leq c$. For any $n\geq 18$, 
$\lfloor \frac{n}{6} \rfloor\leq\frac{n-3}5$, thus the sequence is
decreasing on the domain $s\le k\leq \lfloor \frac{n}6 \rfloor$.
%
%
%This value is less than or equal to $1$ when $\frac{k+1}{n-k+1}$ is
%less than or equal to $c = 2^{-\frac{r}{r-1}}$, and in particular when
%$k \le \frac{c(n+1)}{c+1}$.  Since $\frac12 \le c \le 1$, we have
%$\frac{c(n+1)}{c+1} \ge \frac n4$, and the sequence is decreasing on
%the domain $s\le k\leq \lfloor n/4 \rfloor$.
% 
%It follows that 
% 
\begin{eqnarray*}
    \sum_{k=s+1}^{\lfloor n/6\rfloor} \frac{I_{n-k}^r}{(n-k)!}\,
    \frac{I_k^r}{k!} &\le& (\lfloor n/6\rfloor- s )
    \frac{I_{s+1}^r}{(s+1)!}\, \frac{I_{n-s-1}^r}{(n-s-1)!}
\end{eqnarray*}
From Lemma \ref{minomajo}, $I_{n-s-1}^r\leq
\frac{1}{(n-s+1)^r}\,I_{n-s}^r$ and $I_{s+1}^r\leq (s+1)^r
e^{r/\sqrt{s+1}} \,I_{s}^r.$ Therefore
\begin{eqnarray*}
    \frac{I_{n-s-1}^r}{(n-s-1)!}\, \frac{I_{s+1}^r}{(s+1)!} &\le&
    \frac{(s+1)^r e^{r/\sqrt{s+1}}\ I_{s}^r\, I_{n-s}^r}{(s+1)!\
    (n-s-1)!\ (n-s+1)^r}\\
    &\le& \frac{I_{n-s}^r}{(n-s)!}\ \frac{I_{s}^r\, (s+1)^{r-1}\,
    e^{r/\sqrt{s+1}}}{s!\ (n-s+1)^{r-1}}
\end{eqnarray*}
\noindent and we have
$$\sum_{k=s+1}^{\lfloor n/6\rfloor} \frac{I_{n-k}^r}{(n-k)!}\,
\frac{I_k^r}{k!} \enspace\le\enspace (\lfloor n/6\rfloor - s )
\frac{I_{n-s}^r}{(n-s)!}\ \frac{I_{s}^r\, (s+1)^{r-1}\,
e^{r/\sqrt{s+1}}}{s!\ (n-s+1)^{r-1}},$$
which is $\O\left(\frac{I_{n-s}^r}{(n-s)!}n^{-(r-2)}\right)$.  Note
that it is $\O\left(\frac{I_{n-s}^r}{(n-s)!}\right)$ when $r=2$.

%%%%%%%%%%%%%%%%%%%

We now study the remaining part of the sum $S$.  It follows from Lemma
\ref{minomajo} that if $k\ge s$, then $I_{n-k} \leq
\frac{(n-(k-1))!}{(n-(s-1))!} I_{n-s}$, so we have
\begin{eqnarray*}
    \frac{I_{n-k}^r}{(n-k)!} &\leq&
    \frac{(n-(k-1))!^{r-1}}{(n-(s-1))!^{r-1}} \frac{(n-k+1)}{(n-s+1)}
    \frac{I_{n-s}^r}{(n-s)!}\\
    &=&\frac{(n-k)!^{r-1}}{n!^{r-1}} (n-k+1)^{r-1} \prod_{i=0}^{s-2}
    (n-i)^{r-1} \,\frac{(n-k+1)}{(n-s+1)} \frac{I_{n-s}^r}{(n-s)!}.
\end{eqnarray*}
By Lemma \ref{minomajo} again, we have $I_{k} \leq k!\
e^{2\sqrt{k}-1}$, and hence
$$\frac{I_{k}^r}{k!} \leq k!^{r-1} e^{(2\sqrt{k}-1)r}$$
It follows that
$$\frac{I_{n-k}^r}{(n-k)!} \, \frac{I_k^r}{k!} \leq
\binom{n}{k}^{-(r-1)} \frac{(n-k+1)^r}{n-s+1} \prod_{i=0}^{s-2}
(n-i)^{r-1} e^{(2\sqrt{k}-1)r} \frac{I_{n-s}^r}{(n-s)!}.$$
Note that for $s=1$, $\prod_{i=0}^{s-2} (n-i)^{r-1}=1$.

Now observe that for each $s \leq k \leq n/2$,
$\frac{(n-k+1)^r}{n-s+1}\prod_{i=0}^{s-2} (n-i)^{r-1} < n^{s(r-1)}$.
We get
\begin{eqnarray*}
    \sum_{k=\lfloor n/6\rfloor+1}^{\lfloor n/2\rfloor}
    \frac{I_{n-k}^r}{(n-k)!}\, \frac{I_k^r}{k!} &\leq& n^{s(r-1)}\,
    \frac{I_{n-s}^r}{(n-s)!}\sum_{k=\lfloor n/6\rfloor+1}^{\lfloor
    n/2\rfloor} \binom{n}{k}^{-(r-1)} e^{(2\sqrt{k}-1)r},
    %\textrm{ that is,}\\
    % &\leq&  2 \, j_{n-s} \, \left( \frac{I_s^r}{s!}+ \frac{n^{s(r-1)}}{2^{r}} \,
    %\sum_{k=s+1}^{\lfloor n/2\rfloor} \binom{n}{k}^{-(r-1)}
    %(k+1)^{2r}\right).
\end{eqnarray*}
For any $k$ such that $\lfloor n/6\rfloor+1\le k \le \lfloor
n/2\rfloor$ we have $\binom{n}{k}\geq\binom{n}{\lfloor
  n/6\rfloor}$. Using Stirling formula ($n!\sim \sqrt{2 \pi} e^{-n}
n^{n+\frac12}$) we get
\begin{eqnarray*}
    \binom{n}{\lfloor n/6\rfloor}^{-1}&\sim& \sqrt{2 \pi}
    \frac{\lfloor n/6\rfloor^{\lfloor n/6\rfloor+1/2}}{n^{n+1/2}}
    (n-\lfloor n/6\rfloor)^{n-\lfloor n/6\rfloor+1/2}\\
    &\sim& \sqrt{2\pi n} \left(\frac{\lfloor
    n/6\rfloor}{n}\right)^{\lfloor n/6\rfloor+1/2}
    \left(1-\frac{\lfloor n/6\rfloor}{n}\right)^{n-\lfloor
    n/6\rfloor+1/2}
\end{eqnarray*}
Since $\frac{\lfloor n/6\rfloor}n \le \frac16$ and $1-\frac{\lfloor
n/6\rfloor}{n} < 1$, there exists  $0<C<1$ such that
$\binom{n}{\lfloor n/6\rfloor}^{-1}\leq C^n$ for $n$ large enough.
Hence,
\begin{eqnarray*}
    \sum_{k=\lfloor n/6\rfloor+1}^{\lfloor n/2\rfloor}
    \binom{n}{k}^{-(r-1)} e^{(2\sqrt{k}-1)r} &\leq&
    \frac{n}{3}\,\binom{n}{\lfloor
    n/6\rfloor}^{-(r-1)}e^{(2\sqrt{\lfloor n/2\rfloor}-1)r} \\
    &\leq& \frac{n}3\  C^{(r-1)n}\ e^{r\sqrt{2n}}
\end{eqnarray*}    
In particular, for any $D$ such that $C<D<1$ we have
$$\sum_{k=\lfloor n/6\rfloor+1}^{\lfloor n/2\rfloor}
\binom{n}{k}^{-(r-1)} e^{(2\sqrt{k}-1)r} =
\O\left(D^{(r-1)n}\right).$$
%for any $k$
%  such that $\lfloor n/4\rfloor+1\le k \le \lfloor n/2\rfloor$, we
%  have $$\binom{n}{k}^{-(r-1)}= \O\left(C^{(r-1)n} \right)$$
Consequently
$$\sum_{k=\lfloor n/6\rfloor+1}^{\lfloor n/2\rfloor}
\frac{I_{n-k}^r}{(n-k)!}\, \frac{I_k^r}{k!} = 
\frac{I_{n-s}^r}{(n-s)!}\ \O\left(n^{s(r-1)}
D^{(r-1)n} \right) \quad \textrm{where }
0<D<1$$
%     
%Consider for a moment the coefficient $\binom{n}{k}^{-(r-1)}$.  By
%Stirling's formula ($n!\sim \sqrt{2 \pi} e^{-n} n^{n+1/2}$) we get
%\begin{eqnarray*}
%\binom{n}{k}^{-1}&\sim& \sqrt{2 \pi} \frac{k^{k+1/2}}{n^{n+1/2}} (n-k)^{n-k+1/2}\\
%&\sim& \sqrt{2\pi k} \left(\frac{k}{n}\right)^k \left(1-\frac{k}{n}\right)^{n-k+1/2}
%\end{eqnarray*}
%Moreover when $\lfloor n/4\rfloor+1\le k \le \lfloor n/2\rfloor$ setting $k=\alpha n$, we have
%\begin{eqnarray*}
%\binom{n}{\alpha n}^{-1}&=& \Theta \left(n^{1/2} \left(\alpha^{\alpha} (1-\alpha)^{1-\alpha}\right)^n \right)
%\end{eqnarray*}
%and setting $C=\max_{1/4 \leq \alpha\leq 1/2} \alpha^{\alpha} (1-\alpha)^{1-\alpha}$,
%where $0<C<1$, for $\lfloor n/4\rfloor+1\le k \le \lfloor n/2\rfloor$, we
%have
% 
%$$\binom{n}{k}^{-(r-1)}= \Theta\left(n^{(r-1)/2}C^{(r-1)n} \right)$$
%
%Consequently
%$$\sum_{k=\lfloor n/4\rfloor+1}^{\lfloor n/2\rfloor} \frac{I_{n-k}^r}{(n-k)!}\, \frac{I_k^r}{k!} 
%    = \Theta\left(n^{(s+1/2)(r-1)+1} C^{(r-1)n} \right) \quad \textrm{where } 0<C<1$$
%\note{pas compris la raison de ce d\'ecoupage}
%
Finally  we find that
\begin{eqnarray*}
    S = \sum_{k=s}^{n-s}\frac{I_k^r}{k!}\, \frac{I_{n-k}^r}{(n-k)!} &=& 2 \frac{I_{n-s}^r}{(n-s)!}
    \left(\frac{I_s^r}{s!} + \O\left(n^{-(r-2)}\right) +
    \O\left(n^{s(r-1)} D^{(r-1)n} \right) \right) \\
    &=& 2 \frac{I_{n-s}^r }{ (n-s)!} \O\left(1\right).
\end{eqnarray*}    
Hence $\sum_{k=s}^{n-s} j_k
j_{n-k} = \O (j_{n-s})$.  Thus we can apply Theorem \ref{thbender} for
any fixed positive integer $s$, it yields
\begin{eqnarray*}
  c_n &=& \sum_{k=0}^{s-1}d_k j_{n-k} + \O\left(j_{n-s}\right).\\  
\end{eqnarray*}
Since $d_0=1$ and $d_1 = -j_1=-I_1^r=-2^r$, we get that    
\begin{eqnarray*}
  c_n &=& j_n - 2^r j_{n-1} +\O\left(j_{n-2}\right).\\  
\end{eqnarray*}
Now Equation (\ref{eqj}) above yields
$$j_{n-2}=\O\left(\frac{j_{n-1}}{n^{r-1}}\right) =
\O\left(\frac{j_n}{n^{2(r-1)}}\right),$$
and the independent technical Proposition \ref{prop cor du col} below
yields
$$\frac{I_n}{n!}=\frac{e^{-1/2}}{2\sqrt{\pi}}\,n^{-1/4} e^{2\sqrt{n}}
(1+o(1)).$$
Therefore 
$$\frac{I_{n-1}}{I_n}=\frac{1}{n} \left(1-\frac{1}{n}\right)^{-1/4}
e^{2\sqrt{n} \left(\sqrt{1-\frac{1}{n}}-1\right)} (1+o(1)) =
\frac{1}{n} (1+o(1)) $$
and 
$$\frac{I_{n-1}^r}{I_n^r}= \frac{1}{n^r} (1+o(1)).$$
Finally
$$j_{n-1}=\frac{I_{n-1}^r}{(n-1)!}= \frac{I_{n}^r}{n!}\,
\frac{1}{n^{r-1}}\,(1+o(1))= \frac{j_{n}}{n^{r-1}}\,(1+o(1)).$$
We conclude that    
$c_n=j_n \left(1 - \frac{2^r}{n^{r-1}} +
o\left(\frac{1}{n^{r-1}}\right)\right)$ or
\begin{equation}\label{eqc}
    C_n \enspace=\enspace I_n^r \left(1 - \frac{2^r}{n^{r-1}} +
    o\left(\frac{1}{n^{r-1}}\right)\right).
\end{equation}
Recall that we denote by $p_n$ the probability for an $n$-vertex graph
whose transitions are defined by an $r$-tuple of partial injections on
$\inter[1,n]$ to be connected.  Equation~\ref{eqc} shows that
$$p_n \enspace=\enspace \frac{C_n}{I_n^r} \enspace=\enspace 1 -
\frac{2^r}{n^{r-1}} + o\left(\frac{1}{n^{r-1}}\right),$$
which concludes the proof of Theorem~\ref{asympt connected}.
\eopo

\subsection{1-trimness}\label{sec:trim}

Recall that for a labeled $A$-graph $\Gamma$ to be admissible,
$\Gamma$ must be 1-trim, that is, no vertex $v\ne 1$ may be a leaf
(see Section~\ref{sec:rejection}, and also Section~\ref{sec:
representation}).  Moreover, vertex $v$ is a leaf if and only if it
has less than 2 images or preimages in the $|A|$ partial injections
that define $\Gamma$.

In this section, we show the following result.

\begin{theorem}\label{thm: trimness}
    Let $A$ be an alphabet of cardinality at least 2.  The probability
    for an $n$-vertex labeled $A$-graph to have no leaf is $1+o(1)$.
\end{theorem}

We immediately record the following corollary, which gives our random 
generation strategy its final justification.

\begin{corollary}\label{admissibility}
    Let $A$ be an alphabet of cardinality at least 2.  The probability
    for a given size $n$ labeled $A$-graph $\Gamma$ to be admissible,
    is $1+o(1)$.
\end{corollary}

\begin{proof}
  In view of Theorems \ref{asympt connected} and \ref{thm: trimness},
  an $n$-vertex labeled $A$-graph is connected with probability
  $1+\O\left(\frac{1}{n}\right)$, and without leaves (in particular: 
  1-trim) with probability $1 + o(1)$. It follows that the probability for
  an $n$-vertex labeled $A$-graph to be connected and without leaves is
  $$\left(1+\O\left(\frac{1}{n}\right)\right)\ \left(1+o(1)\right) =
  1+o(1),$$
  and the result follows.
\end{proof}

The rest of this section is devoted to the proof of Theorem~\ref{thm:
trimness}.  We first observe that it suffices to establish this result
when $|A| = 2$: indeed, if $A = \{a_1,\ldots,a_r\}$ with $r > 2$ and
if the $\{a_1,a_2\}$-part of an $A$-graph has no leaf, then neither
does the whole $A$-graph.  So we assume that $r = 2$ for the rest of
this section.

%%%%%%%%%%%%%%%%%%%%%%%%
\subsubsection{Number of sequences}\label{sec: nber of seq}

We note that a leaf in an $A$-graph forms a length 1 sequence in
the functional graph of one of the partial injections, and is an
endpoint of a sequence in the functional graph of the other.

We start with a study of the parameter $X_n$, which counts the number
of sequences in the functional graph of a partial injection on
$\inter[1,n]$.  We will compute the expectation and the variance of
$X_n$ in Section~\ref{sec e and sv of Xn}.  Note that the number of
endpoints of sequences in the functional graph of a random partial
injection is bounded above by the quantity accounted for by the
parameter $2X_n$.

We first introduce the bivariate series $J(z,u) = \sum_{n,k}
\frac{J_{n,k}}{n!} z^nu^k$, where $J_{n,k}$ denotes the number of
partial injections on $\inter[1,n]$, whose functional graph has $k$
sequences (that is, such that $X_n = k$). The EGS of non-empty 
sequences was already computed and it is equal to $\frac z{1-z}$ or, 
in this multivariate setting, $\frac{zu}{1-z}$. The EGS of cycles is 
$\log\frac z{1-z}$ (also in this multivariate setting). The 
multivariate analogue of Proposition~\ref{EGS calculus} (see Flajolet 
and Sedgewick \cite[ex. 7, section III.3]{FSBook}) then shows that
$$J(z,u) = \exp\left(\frac{zu}{1-z} + \log\left(\frac1{1-z}\right)\right) =
\frac{1}{1-z}\exp\left(\frac{zu}{1-z}\right).$$
In particular, $J(z,1) = I(z)$. 

The expectation of variables $X_n$ and $X_n^2$ are
$$\E(X_n) = \frac{\sum_k kJ_{n,k}}{\sum_k J_{n,k}} \quad \hbox{and}
\quad
\E(X_n^2) = \frac{\sum_k k^2J_{n,k}}{\sum_k J_{n,k}}.$$
Now $\sum_k J_{n,k}$ is the coefficient of $z^n$ in $J(z,1)$ and
$\sum_k kJ_{n,k}$ is the coefficient of $z^n$ in
$\frac{\partial}{\partial u}J(z,u)\big|_{u=1}$.  The coefficient of
$z^n$ in $\frac{\partial^2}{\partial u^2}J(z,u)\big|_{u=1}$ is $\sum_k
k(k-1)J_{n,k}$.

Let
$$H_p(z) = \frac{1}{(1-z)^{p}} I(z) = \frac{1}{(1-z)^{p+1}}
\exp\left(\frac{z}{1-z}\right).$$
Then we have
\begin{eqnarray*}
    J(z,1) &=& I(z) \enspace=\enspace H_0(z),\\
    \frac{\partial}{\partial u} J(z,u) \Big|_{u=1} &=&
    \frac{z}{(1-z)^2}\exp\left(\frac{z}{1-z}\right) \enspace=\enspace
    \frac{z}{1-z} I(z) \enspace=\enspace z H_1(z),\\
    \frac{\partial^2}{\partial u^2} J(z,u) \Big|_{u=1} &=&
    \frac{z^2}{(1-z)^3}\exp\left(\frac{z}{1-z}\right)
    \enspace=\enspace \frac{z^2}{(1-z)^2} I(z) \enspace=\enspace z^2 H_2(z).\\
\end{eqnarray*}
For convenience, if $S(z)$ is a formal power series, we let
$[z^n]S(z)$ denote the coefficient of $z^n$ in $S(z)$.  Then the
expectation of variables $X_n$ and $X^2_n$ are given by

\begin{eqnarray}\label{eq:meanx}
    \E(X_n) &=& \frac{[z^n]z H_1(z)}{[z^n] I(z)}= \frac{[z^{n-1}]
    H_1(z)}{[z^n] H_0(z)},\\
    \label{eq:meanx2}
    \E(X_n^2) &=& \frac{[z^n] z^2H_2(z)}{[z^n] I(z)} +
    \frac{[z^n]zH_1(z)}{[z^n]I(z)}=\frac{[z^{n-2}] H_2(z)}{[z^n]
    H_0(z)} + \frac{[z^{n-1}]H_1(z)}{[z^n]H_0(z)}.
\end{eqnarray}
The variance of $X_n$ is
\begin{eqnarray}\label{eq:varx}
\sigma^2(X_n) &=& \E(X_n^2) - \E(X_n)^2.
\end{eqnarray} 

Thus finding asymptotic estimates of $\E(X_n)$ and $\sigma^2(X_n)$
requires finding estimates of the coefficients of the functions
$H_p(z)$ for $p = 0,1,2$.

%%%%%%%%%%%%%%%%%%%%%%%%
\subsubsection{Saddlepoint asymptotics} \label{sec:saddlepoint} 

Saddlepoint analysis is a powerful method to find asymptotic estimates
of the coefficients of analytic functions which exhibit
exponential-type growth in the neighborhood of their singularities.
We refer the reader to the books by Flajolet and Sedgewick
\cite{FSBook1} and \cite[Chap.  VIII]{FSBook}, and to the survey by
Odlyzko \cite{odlyzko} for a thorough presentation of saddlepoint
analysis.

The fast growth of the coefficients of $H_p(z)$ justifies the
application of saddlepoint analysis.  The theorem we want to use,
Theorem~\ref{thcol} below, requires an additional hypothesis, namely
the H-admissibility of the functions $H_p$ \cite[Section
VIII.5]{FSBook}.  We now verify that this rather technical condition
is satisfied.

Let $f(z)$ be a function that is analytic at the origin, with radius
of convergence $\rho$, positive on  $]0,\rho[$. Put $f(z)$
into its exponential form $f(z)=e^{h(z)}$ and let
$$a(r)=rh'(r) \quad \mbox{and} \quad b(r)=r^2h^{''}(r)+rh'(r).$$
The function $f(z)$ is said to be {\it H-admissible} if there exists a
function $\delta\colon ]0,\rho[ \longrightarrow ]0,\pi[$ such that the
following three conditions hold:
\begin{itemize}
    \item[(H1)] $\lim_{r \rightarrow \rho} b(r) = +\infty.$
    \item[(H2)] Uniformly for $|\theta| \leq \delta(r)$
    $$f(re^{i\theta}) \sim f(r)e^{i\theta a(r)
    -\frac{1}{2} \theta^2 b(r)}\quad\hbox{when $r$ tends to $\rho$.}$$
    [That is, $f(re^{i\theta}) = f(r)e^{i\theta a(r)
    -\frac{1}{2} \theta^2 b(r)} (1 + \gamma(r,\theta))$ with 
    $|\gamma(r,\theta)| \le \tilde\gamma(r)$ when $|\theta| < 
    \delta(r)$ and $\lim_{r \rightarrow \rho}\tilde\gamma(r) = 0$.]
    \item[(H3)] and uniformly for $\delta(r)\leq |\theta| \leq \pi$ 
    $$f(re^{i\theta})\sqrt{b(r)} = o(f(r)) \qquad \mbox{when $r$ 
    tends to $\rho$.}$$
\end{itemize}

\begin{lemma}
    The functions $H_0(z)$, $H_1(z)$ and $H_2(z)$ are $H$-admissible.
\end{lemma}    

\begin{proof}
    First it is elementary to verify that
    $H_p(z)=\frac{1}{(1-z)^{p+1}} \exp\left(\frac{z}{1-z}\right)$ ($p
    = 0,1,2$) is analytic at the origin, with radius of convergence
    $\rho=1$, and is positive on the real segment $]0,\rho[$.
    Following the definition of H-admissibility above, we find that
    $H_p(z) = e^{h_p(z)}$ with $h_p(z) = \frac z{1-z} -
    (p+1)\log(1-z)$, so that
\begin{equation}\label{eq:apbp}
a_p(r)=r \frac{(p+2) - (p+1)r}{(1-r)^2} \quad \mbox{and} \quad
b_p(r)=r\frac{(p+2)-pr}{(1-r)^3}.
\end{equation}
Therefore Condition (H1) is satisfied.

Let $\delta(r)=(1-r)^{17/12}$ (for a discussion on the choice of
$\delta$, we refer the readers to \cite[Chap. VIII]{FSBook}).
For $\theta$ small enough, one can expand $h_p(re^{i\theta})$ into
$$h_p(re^{i\theta}) = h_p(r)+\sum_{m=1}^{\infty}\alpha_m(r)
\frac{(i\theta)^m}{m!},$$
with $\alpha_{m}(r)=r \frac{d}{dr}\alpha_{m-1}(r)$ and $\alpha_0(r) =
h_p(r)$ (by definition of a Taylor development).  In particular, we
find
\begin{eqnarray*}
    \alpha_1(r) &=& a_p(r), \\
    \alpha_2(r) &=& b_p(r),\quad\hbox{and}\\
    \alpha_3(r) &=& \frac{r}{(1-r)^4}\left((2+p) + 4r - pr^2\right).
\end{eqnarray*}
Thus, as $r$ tends towards $1$, for $|\theta|\leq \delta(r)$,
$\alpha_3(r)\theta^3 = \O\left((1-r)^{1/4}\right) = o(1)$.  More
generally, $\alpha_m(r)\theta^m = \O\left((1-r)^{1/{m+1}}\right) =
o(\alpha_{m-1}(r))$.  Therefore, uniformly for $|\theta| \le
\delta(r)$
$$h_p(re^{i\theta})= h_p(r)+i\theta a(r) -\frac{1}{2} \theta^2 b(r) +
o(1)$$
and Condition (H2) follows.

Finally, we have
\begin{eqnarray*}
    |H_p(re^{i\theta})| & = & \frac{1}{|1-re^{i\theta}|^{p+1}} \exp
    \bigg( \Re \bigg( \frac{re^{i\theta}}{1-re^{i\theta}}
    \bigg)\bigg)\\
    &=& \frac{1}{(1+r^2-2r\cos \theta)^{(p+1)/2}} \exp \bigg(
    \frac{r(\cos \theta - r)}{1+r^2-2r\cos \theta}\bigg).
\end{eqnarray*}
We observe that for $r>0$, $r(\cos\theta - r)$ and $(1 + r^2 
-2r\cos\theta)^{-1}$ are decreasing functions of $\theta$ on 
$]0,\pi[$. Thus, for each $r > 0$ and $\delta(r) \le |\theta| < \pi$, 
$|H_p(re^{i\theta})|$ is bounded above by $|H_p(re^{i\delta(r)})|$, 
namely
$$
    |H_p(re^{i(1-r)^\frac{17}{12}})| =
    \frac{\exp \bigg(\frac{r(\cos (1-r)^\frac{17}{12} -
    r)}{1+r^2-2r\cos (1-r)^\frac{17}{12}}\bigg)}{(1+r^2-2r\cos
      (1-r)^\frac{17}{12})^{(p+1)/2}}
$$    

In the neighborhood of 0, $\cos\theta = 1 - \frac12 \theta^2 +
\O(\theta^4)$.  So when $r$ tends towards $1$, we have
\begin{eqnarray*}
    \cos (1-r)^{17/12} &=& 1 -\frac{1}{2} (1-r)^{17/6} +
    \O\left((1-r)^{17/3}\right),\\
    1+r^2-2r\cos(1-r)^{17/12}&=&1+r^2-2r +r (1-r)^{17/6} +
    \O\left((1-r)^{17/3}\right)\\
    &=&(1-r)^2 \left(1+r (1-r)^{5/6} +
    \O\left((1-r)^{11/3}\right)\right),\\
    \frac{1}{1+r^2-2r\cos(1-r)^{17/12}} &=& 
    \frac{(1-r)^{-2}}{1+r (1-r)^{5/6} + \O\left((1-r)^{11/3}\right)}\\
    &=& (1-r)^{-2} \left(1-r (1-r)^{5/6} + 
    \O\left((1-r)^{5/3}\right)\right).
\end{eqnarray*}
Moreover
\begin{eqnarray*}
   \frac{r(\cos (1-r)^\frac{17}{12} -
    r)}{1+r^2-2r\cos (1-r)^\frac{17}{12}} & = & \frac{r}{(1-r)^2} 
   \left(1-r -\frac{1}{2} (1-r)^{17/6} +
   \O\left((1-r)^{17/3}\right)\right)\\
   && \hskip1cm \left(1-r (1-r)^{5/6} + \O\left((1-r)^{5/3}\right)\right)\\
   & = & \frac{r}{1-r} \left(1- \frac12(1-r)^{11/6} 
   +\O\left((1-r)^{14/3}\right)\right)\\
   && \hskip1cm \left(1-r (1-r)^{5/6} + \O\left((1-r)^{5/3}\right)\right)\\
   & = & \frac{r}{1-r} \left(1-(1-r)^{5/6} +\O\left((1-r)^{5/3}\right)\right).
\end{eqnarray*}
It follows that
\begin{eqnarray*}
    |H_p(re^{i(1-r)^{17/12}})| & =& (1-r)^{-(p+1)} \left(1-r
    (1-r)^{5/6} + \O\left((1-r)^{5/3}\right)\right)^{(p+1)/2}\\
    && \hskip.55cm \exp \left(\frac{r}{1-r} \left(1-(1-r)^{5/6}
    +\O\left((1-r)^{5/3}\right)\right) \right).
\end{eqnarray*}
As a result,
\begin{eqnarray*}
    |H_p(re^{i\theta})| &\leq& (1-r)^{-(p+1)} (1+\O(1-r)^{5/6})\\
    && \hskip1cm\exp \Bigg( \frac{r}{1-r} \bigg(1-(1-r)^{5/6}
    +\O\left((1-r)^{5/3}\right) \bigg)\Bigg),\\
    |H_p(re^{i\theta})| &\leq& \left(H_p(r) + \exp\left(\frac
    r{1-r}\right) \O((1-r)^{\frac56 -(p+1)})\right)\\
    && \hskip3cm\exp\left(-\frac
    r{(1-r)^\frac16} + \O((1-r)^\frac23)\right),
\end{eqnarray*}
and hence
$$|H_p(re^{i\theta})|\sqrt{b_p(r)} = o(H_p(r)),$$
that is, Condition (H3) is satisfied.
\end{proof}

We now want to use the following theorem \cite[Theorem VIII.5]{FSBook}.

\begin{theorem}[coefficients of H-admissible functions] \label{thcol}
    Let  $f(z)$ be a H-admissible function and  $\zeta=\zeta(n)$ be the
    unique solution in the interval  $]0,\rho[$ of the saddlepoint equation
    $$\zeta \frac{f'(\zeta)}{f(\zeta)}=n.$$
    Then
    $$[z^n]f(z) = \frac{f(\zeta)}{\zeta^n \, \sqrt{2 \pi
    b(\zeta)}}\left(1+o(1)\right).$$
    where $b(z) = z^2h''(z)+zh'(z)$ and $h(z) = \log f(z)$.
\end{theorem}

Let us record immediately an application of this result.

\begin{proposition}\label{prop cor du col}
    With the notation above, for $p = 0,1,2$, 
    $$[z^n]H_p(z) = \frac{e^{-1/2}}{2\sqrt{\pi}}\,n^{p/2-1/4}
    e^{2\sqrt{n}} (1+o(1)).$$
\end{proposition} 

\begin{proof}
    For any positive integer $n$ the saddle point $\zeta_p(n)$ is the
    least positive solution of $z \frac{H'_p(z)}{H_p(z)} = n$, that
    is, the least positive solution of
    $$(n+p+1)z^2-(2n+p+2)z+n = 0$$
    and it follows that
    $$\zeta_p(n) = 1 - \frac{p + \sqrt{4n+(p+2)^2}}{2(n+p+1)} = 1
    - \frac1{\sqrt{n}} -\frac{p}{2n} + \O\left(\frac1{n^{3/2}}\right).$$
    Moreover from Equation (\ref{eq:apbp}), the function $b_p(r)$ is $
    b_p(r)=r\frac{(p+2)-pr}{(1-r)^3}$.  
    Therefore
    $$[z^n]H_p(z) = \frac{e^{\frac{\zeta_p}{1- \zeta_p}}}{\sqrt{2\pi
    \left((p+2)-p\zeta_p\right)}} \, \frac{1}{\zeta_p^{n}
    (1-\zeta_p)^{p} } \, \sqrt{\frac{1-\zeta_p}{\zeta_p}}
    \left(1+o(1)\right).$$
    Now $\sqrt{\frac{1-\zeta_p}{\zeta_p}} = n^{-\frac14} (1 + o(1))$,
    $(1- \zeta_p)^p = n^{-\frac p2} (1 + o(1))$ and $\sqrt{2\pi
    \left((p+2)-p\zeta_p\right)} = 2\sqrt\pi (1 + o(1))$.  Finally
    $\frac {\zeta_p}{1-\zeta_p} = (\sqrt n - 1 - \frac p2) (1 + o(1))$
    and $\zeta_p^{-n} = \exp(-n\log\zeta_p) = \exp(\sqrt n +
    \frac{p+1}2)(1 + o(1))$.  So
    \begin{eqnarray*}
	[z^n]H_p(z) &=& \frac{ e^{\sqrt{n}-1-p/2}}{2\sqrt{\pi}}\,
	e^{\sqrt{n} +(p+1)/2} n^{p/2-1/4} (1+o(1))\\
	&=& \frac{e^{-1/2}}{2\sqrt{\pi}}\,n^{p/2-1/4} e^{2\sqrt{n}}
	(1+o(1)).
    \end{eqnarray*}
    
\end{proof}    

%%%%%%%%%%%%%%%
\subsubsection{Expected value and standard deviation of $X_n$}\label{sec e and sv of Xn}

We now conclude the study of the expected value and the standard
deviation of the number of sequences in a random partial injection.

\begin{lemma}\label{lm count sequences}
  The expected number $\E(X_n)$ of sequences in a random partial
  injection of size $n$ is asymptotically equal to $\sqrt{n}$ with 
  standard deviation $o(\sqrt{n})$.
\end{lemma}

\begin{proof}
  Recall that the expected values of the random variable $X_n$ and
  $X_n^2$ and the variance of $X_n$ are given in Equations
  (\ref{eq:meanx}), (\ref{eq:meanx2}) and (\ref{eq:varx}) in
  Section~\ref{sec: nber of seq}.  We can use Proposition~\ref{prop
  cor du col} to estimate these quantities.
  
  From Equation (\ref{eq:meanx})  we find that
  $$\E(X_n) = \frac{(n-1)^{1/4}
  e^{2\sqrt{n-1}}}{n^{-1/4}e^{2\sqrt{n}}}(1+o(1)) = \sqrt{n}
  (1+o(1)),$$
  that is, the expected number $\E(X_n)$ of sequences in a random
  partial injection of size $n$ is asymptotically equal to $\sqrt{n}$.
  
  Similarly, from Equation (\ref{eq:meanx2}) we get
  $$\E(X_n^2) = \frac{(n-2)^{3/4}
  e^{2\sqrt{n-2}}}{n^{-1/4}e^{2\sqrt{n}}}(1+o(1)) + \sqrt n (1 + o(1))
  = n (1+o(1)),$$
  and from Equation (\ref{eq:varx}) $ \sigma^2(X_n)=o(n)$.
  Thus the standard deviation $\sigma(X_n)$ of the number of sequences
  in a random partial injection of size $n$ is $o(\sqrt n)$.
\end{proof}

%%%%%%%%%%%%%%
\subsubsection{Proof of Theorem~\ref{thm: trimness}}

Before we prove Theorem~\ref{thm: trimness}, let us recall the
statement of Chebyshev's inequality.

\begin{proposition}[Chebyshev's inequality]    
    If $X$ is a random variable of expectation $\E(X)$ with a finite
    variance $\sigma^2(X)$ then for any positive real $\alpha$
    \begin{equation}\label{eq:tchebychev}
    \mathbb{P}\{|X-\E(X)|\geq \alpha \}< \frac{\sigma^2(X)}{\alpha^2}.
    \end{equation}
\end{proposition}

As we already noted, given a pair of partial injections, a leaf in the
resulting $A$-graph is an endpoint of a sequence in one of the partial
injections, and a singleton in the other.
In view of Lemma~\ref{lm count sequences}, Chebyshev's inequality
(Equation~(\ref{eq:tchebychev})) applied to $X_n$ and $\alpha = \sqrt
n$, shows that
\begin{equation}\label{probaseq}
      \mathbb{P}\{|X_n-\sqrt{n} |\geq \sqrt{n} \} = o(1).
\end{equation}
In other words, the probability that a partial injection of size $n$
contains more than $2\sqrt n$ sequences tends towards $0$ when $n$
tends towards $\infty$.  Let us call such a partial injection
\textit{sequence-rich}.  Then $n$-vertex labeled $A$-graphs defined by
a pair of partial injections, one of which at least is sequence-rich,
occur with probability $o(1)$.
  
Let us now focus on the labeled $A$-graphs defined
by partial injections, each of which contains less than $2\sqrt{n}$
sequences.  Again, a vertex is a leaf if it is an endpoint of
sequence for one of the injections and a singleton in the other one.
As the two injections play symmetric r\^oles, we estimate the
probability for a vertex to be an endpoint in the first injection and
a singleton for the second one: the total estimate will be bounded
above by twice that probability.
  
Since the first injection has at most $2\sqrt n$ sequences, it has at
most $4 \sqrt{n}$ endpoints.  To estimate the number of second partial
injections in which at least one of these vertices is a singleton, we
count the number of partial injections that do not involve such a
vertex.  For any given vertex, there are $I_{n-1}$ such partial
injections, and as we have to consider up to $4\sqrt{n}$ potential
endpoints, the number of partial injections in which at least one of
these vertices is a singleton is bounded above by $4\sqrt{n}I_{n-1}$.
  
Therefore there are at most $8\sqrt{n}I_{n-1}I_n$ pairs of
partial injections that are not sequence-rich and that exhibit at 
least a leaf.  As $I_{n-1}\leq \frac{I_n}{n}$, this number is
less than or equal to $\frac{8}{\sqrt{n}}I_{n}^2$, and the
associated probability is less than or equal to
$\frac{8}{\sqrt{n}}$.

Consequently the probability for an $n$-vertex labeled $A$-graph to
have at least a leaf is less than or equal to
$$ o(1) + \frac{8}{\sqrt{n}} = o(1),$$
which concludes the proof.

%%%%%%%%%%%%%%
\subsection{The number of size $n$ subgroups}

Let $S_{n,r}$ be the number of size $n$ subgroups of $F = F(A)$, where
$\rank(F) = r$.  By Corollary~\ref{admissibility}, the number of
admissible labeled $A$-graphs of size $n$ is $I_n^r (1+o(1))$.  By
Lemma~\ref{fact labeled}, each size $n$ subgroup is represented by
$(n-1)!$ distinct admissible labeled $A$-graphs, so
$$S_{n,r} \sim \frac{I_n^r}{(n-1)!} .$$
Proposition~\ref{prop cor du col} gives us
an equivalent of $I_n/n!$, and it follows that
$$S_{n,r} \sim n\ n!^{r-1} \frac{I_n^r}{n!^r} \sim n\ n!^{r-1}
\frac{e^{-r/2}}{2^r\pi^{r/2}}\,n^{-r/4}
    e^{2r\sqrt{n}} .$$
By Stirling's formula, $n!$ is equivalent to $\sqrt{2\pi} 
e^{-n}n^{n+\frac12}$ and it follows that
% 
%$$S_{n,r} \sim e^{-\frac r2} 2^{-\frac{r+1}2} \pi^{-\frac12} e^{2r\sqrt 
%n-(r-1)n} n^{(r-1)n + \frac{r+2}4}.$$
$$S_{n,r} \sim \frac{(2e)^{-r/2}}{\sqrt{2\pi}} e^{-(r-1)n+2r\sqrt{n}}n^{(r-1)n+\frac{r+2}4}.$$

%%%%%%%%%%%%%%%%
\section{Random generation algorithm}\label{sec:randomgeneration}

As discussed in Section~\ref{sec: enum}, and in particular in
Section~\ref{sec:rejection} (see also Section~\ref{sec:algoadmissible}
below), the core of our random generation algorithm for admissible 
$A$-graphs, is a procedure to randomly generate size $n$ partial
injections.

The recursive decomposition of partial injections investigated so far
allows us to use the \textit{recursive method} introduced by Flajolet,
Zimmermann and van Cutsem \cite{FZvC} (following work by Nijenhuis and
Wilf \cite{nw78}) to \textit{efficiently} and \textit{randomly}
generate partial injections of size $n$.

Recall that a partial injection is a set of disjoint 
\textit{components}, that are either cycles or non-empty sequences. 
The recursive method consists, in our case, in:
\begin{itemize}
    \item choosing the size $k$ of a component ($k \in
    \{1,\ldots,n\}$) according to the distribution of the sizes of
    components in a random size $n$ partial injection;
    
    \item choosing whether that size $k$ component is a cycle or a
    sequence -- according to the distribution of these two types among
    size $k$ components;
    
    \item and choosing a size $n-k$ partial injection.
\end{itemize}
We give more details below, on how these steps are performed. The 
result of the procedure is a sequence of symbols of the form 
$\xi_1(k_1)\ldots\xi_r(k_r)$, where $k_1+\cdots+k_r = n$, the $\xi_i$ 
are in $\{\sigma,\kappa\}$, $\sigma(k)$ stands for \textit{sequence 
of size $k$}, and $\kappa(k)$ stands for \textit{cycle of size $k$}. 
Such a sequence represents, in a natural way, an unlabeled size $n$ 
partial injection and the last step of the algorithm consists in 
randomly labeling that partial injection.

Let us now be more precise.

%%%%%%%%%%%%%%%%
\subsection{Partial injections}

The tool to grasp the distribution of the sizes of components in
partial injections is the pointing operator $\Theta$: pointing a
labeled object consists in marking one of its atoms, or equivalently
one of its labels from $\{1,\cdots,n\}$.  Naturally, there are $n$ ways
of pointing an object of size $n$. So if the EGS of  a labeled
combinatorial class  $\C$ is $C(z)=\sum
c_nz^n/n!$, then the pointed class of $\C$, denoted by $\Theta\C$, 
has EGS
$$\Theta C(z) = \sum_{n\geq 0}\frac{nc_n}{n!}z^n =
z\frac d{dz}C(z).$$
If $\C$ is, as in our situation, defined as a set of components of a
class $\D$ with EGS $D(z)$, then $C(z)=\exp(D(z))$ and
\begin{eqnarray*}
    \Theta C(z) = z\frac d{dz}C(z) &=& z\frac
    d{dz}\left(\exp(D(z))\right)\\
    &=& z \left(\frac d{dz}D(z) \right) \exp(D(z)) = \Theta D(z)\
    C(z).
\end{eqnarray*}    
The combinatorial interpretation of this equality is the following:
marking an atom of an element of $\C$ amounts to marking an atom of
one of its components (of size, say, $k$), and the remaining part of
the element of $\Theta\C$ is a non-pointed element of $\C$ of size
$n-k$.

For partial injections, we have $I(z) = \exp(D(z))$, with $D(z) =
\frac{z}{1-z} + \log(\frac1{1-z})$.  Therefore
\begin{equation}\label{eq:pointing}
\Theta I(z) = \Theta D(z) \times I(z) =
\left(\frac{z}{(1-z)^2}+\frac z{1-z}\right) I(z),
\end{equation}
where $\frac{z}{(1-z)^2}$  is the EGS of pointed labeled sequences 
and $\frac{z}{1-z}$  is the EGS of pointed labeled cycles. Now
$$[z^k] \frac{z}{(1-z)^2} = \cases{k & if $k\ge 1$,\cr 0
& if $k = 0$,}
\qquad\qquad
[z^k] \frac{z}{1-z} = \cases{1 & if $k\ge 1$,\cr 0 &
if $k = 0$,}$$
so we have
\begin{equation}\label{eq-generation}
    n\frac{I_n}{n!} = \sum_{k=1}^n(k+1)\frac{I_{n-k}}{(n-k)!}\ .
\end{equation}
Therefore the probability $p_k = \mathbb{P}(\textrm{size $k$})$ for the
pointed component to be of size $k$ is
$$p_k \enspace=\enspace \frac{(k+1)
\frac{I_{n-k}}{(n-k)!}}{\frac{I_n}{(n-1)!}} \enspace=\enspace 
\frac1{I_n}\left((k+1) 
\frac{(n-1)!}{(n-k)!} I_{n-k}\right)$$
and the probability for a size $k$ component to be a sequence (resp. 
a cycle) is
$$\mathbb{P}(\textrm{sequence}) = \frac k{k+1}
\qquad\qquad
\mathbb{P}(\textrm{cycle}) = \frac 1{k+1}.$$

We are now ready to describe the random generation algorithm for
partial injections of size $n$. The discussion of its complexity is 
postponed to Section~\ref{sec: complexity}.

Let $\func{Uniform}{[0,1[}$ be the function that returns a real number
chosen uniformly at random in the interval $[0,1[$.  Recall that if
$X$ is a random variable with values in $\inter[1,n] $ with
probability $\mathbb{P}(X=i)=p_i$ then the value of $X$ can be
generated randomly with respect to this probability distribution as
follows (see \cite{dev86} for example).
    
{\small\begin{tabbing}
\quad \= \qquad \= \qquad \= \qquad \kill
$\textsc{Random}X$ \\
\> \textsf{dice} = $\func{Uniform}{[0,1[}$\\
\> $k = 1$, $S = p_1$ \\
\> \keyw{while} $\textsf{dice} \ge S$ \\
\> \> $k = k+1$ \\
\> \> $S = S + p_k$ \\
\> \keyw{return} $k$ \\
\end{tabbing}}

\vskip -.5cm

Our algorithm to randomly generate a partial injection of size $n$
uses directly this idea.  Because the probabilities discussed above
(for a component of a random partial injection to have size $k$, for a
size $k$ component to be a cycle) are rational numbers, we choose to
express the algorithm entirely in integers, in order to facilitate
exact computation, and thereby to guarantee the absence of bias in the
distribution of the partial injections.  Concretely, we multiply
\textsf{dice} and the $p_k$ by $I_n$.

The algorithm requires a preliminary phase, during which a table
containing the values of $I_k$ ($0 \le k\le n$) is computed
using the recurrence relation in Equation~(\ref{eq:rec}).

We denote by $\func{Uniform}{N}$ a function that  returns an integer
chosen uniformly at random in the interval $[0,N[$.

{\small\begin{tabbing}
\quad \= \qquad \= \qquad \= \qquad \kill
$\func{RandomPartialInjection}{n}$ \qquad \\
\> $Result = []$ \\
\> \keyw{while} $n>0$\\
\> \> \com{Compute the size $k$ of a component}\\
\> \>  \textsf{dice} = $\func{Uniform}{I_n}$\\
\> \>  $k = 1$, $T = 1$, $S = 2{I_{n-1}}$ \quad\com{That is, $S = I_np_1$} \\
\> \> \keyw{while} $\textsf{dice} \ge S$ \\
\> \> \> $k = k+1$ \\
\> \> \> $T = T * (n - k + 1)$ \\
\> \> \> $S = S +  (k+1) T I_{n-k}$ \quad\com{That is, $S = S +  I_np_k$} \\
% \> \> \keyw{return} $k$ \\
\> \> \com{Decide whether the component is a sequence or a cycle}\\
\> \> $\textsf{dice}' = \func{Uniform}{k+1}$\\
\> \> \keyw{if} $\textsf{dice}' < k$\\
\> \> \> \keyw{then} Append $\sigma(k)$ to $Result$\\
\> \> \> \keyw{else} Append $\kappa(k)$ to $Result$\\
\> \> $n = n-k$\\
\> \com{Randomly label the final result}\\
\> Label $Result$ with $\func{RandomPermutation}{n}$\\
\> \keyw{return} $Result$ \\
\end{tabbing}}
 
\vskip -.5cm

The outer \keyw{while} loop of the algorithm produces a sequence of
symbols of the form $(\xi_1(k_1),\ldots,\xi_r(k_r))$, with each
$\xi_i \in \{\sigma,\kappa\}$ and such that $\sum_{i=1}^r k_i = n$,
that describes the size and the nature of the components of the size
$n$ partial injection.

The last step of the algorithm, which randomly generates a permutation
of the $n$ elements on which the partial injection is defined, can be
performed in linear time and space using the algorithm given in
Section \ref{sec:permutations}.

%%%%%%%%%%%%%%%%%%%%%%%%%%%%%%%%%
\subsection{Admissible $A$-graphs}\label{sec:algoadmissible}

Recall (see Section~\ref{sec:rejection}) that our algorithm to
generate admissible $A$-graphs consists in randomly generating, for
each letter of the alphabet $A$, a partial injection of size $n$, and
then using a rejection algorithm to keep only admissible $A$-graphs.

Assuming that the table for the values of $I_n$ is already
computed, the algorithm reads as follows.
{\small\begin{tabbing}
    \quad \= \qquad \= \qquad \= \qquad \kill
$\func{RandomAdmissibleAGraph}{n}$ \qquad \\
\> \keyw{repeat} \\
\> \> \keyw{for each} $a\in A$\\
\> \> \> compute the partial injection $I_a$ for $a$ using
$\func{RandomPartialInjection}{n}$\\
\> \keyw{until} the resulting $A$-graph is admissible.\\
\end{tabbing} }

%%%%%%%%%%%%%%%%%%%%%%%%
\subsection{Complexity}\label{sec: complexity}

The algorithm requires manipulating large integers: $I_n$ is of the
order of $\O(n^n)$.  Computing in multiprecision, that is, without any
approximation, guarantees the absence of bias in the distribution of
the generated objects.  However, we also briefly discuss floating
point implementations at the end of this section.

We first evaluate the complexity of the algorithm in the RAM model,
\textit{i.e.}, under the unit cost assumption (or uniform cost
convention) according to which each data element (here: each integer)
is stored in one unit of space, and the elementary operations
(reading, writing, comparing, performing arithmetic operations, etc)
require one unit of time, -- even for large numbers.

The pre-computation phase that stores the values of $I_n$ uses the
recurrence relation on the $I_n$ given in Equation  (\ref{eq:rec}), and it
requires $\O(n)$ operations.

The (worst case) time complexity $T(n)$ of
$\func{RandomPartialInjection}{n}$ (assuming that the values of the
$I_n$ are precomputed) satisfies the following inequality
$$T(n) \le \max_{1\le k\le n}(ck+ T(n-k))$$
(where $c$ is a constant), so that $T(n) \le cn$, that is, $T(n) =
\O(n)$.

Next, checking whether the $A$-graph generated is connected can be
done using common algorithms on graphs (depth-first search) in time
$\O(n)$.  Checking 1-trimness is also done in $\O(n)$, by scanning the
list of edges which has at most $|A|\,n$ elements.  This part of the
algorithm does not require manipulating large numbers.

Corollary \ref{admissibility} shows that size $n$ $A$-graphs are
admissible with probability $1+o(1)$.  Therefore the number of rejects
(see Section \ref{sec:rejection}) for lack of admissibility is equal
in average to $\frac{1}{1+o(1)} = 1+o(1)$.  Thus the average number of
rejects tends to 0 when $n$ tends to infinity.

In conclusion, for the RAM model, the random generation algorithm
requires a precomputation that can be done in linear time, and it
uses, in average, $\O(n)$ operations to generate each admissible
$A$-graph.

In the case of the bit complexity (or logarithmic cost convention), an
integer $N$ is handled via its binary representation, of length
$\O(\log N)$.  In particular, the representation of $I_n$ has length
$\O(n\log n)$.  The basic operations on numbers of that size (reading,
writing, comparison, addition, multiplication by a number whose binary
representation is of length $\O(n)$) are performed in time $\O(n\log
n)$, whereas the multiplication of two such numbers takes time
$\O(n\log^2n)$.  Under this bit-cost assumption, the time and space
required for the pre-computation are $\O(n^2\log n)$ (instead of
$\O(n)$ in the RAM model).  And each random draw takes time
$\O(n^2\log^2n)$.

\begin{remark}
    Under the bit-cost assumption, we should also take into 
    consideration the complexity of the function $\func{Uniform}{N}$.
    Since this function returns an integer, it can be performed by a 
    rejection algorithm, 
    randomly choosing each bit of (the binary expansion of) $N$, that 
    is, in $1+\lfloor\log n\rfloor$ unit cost operations. In 
    such a process, the probability that the integer generated is 
    greater than $N$ is at most $1/2$ ($\frac{b-1}b$ in base $b$), so 
    the average number of reject is at most 2, and the complexity of 
    $\func{Uniform}{N}$ is $\O(\log N)$.
\end{remark}

In practice, it is often convenient to use floating point arithmetic
instead of multiprecision arithmetic.  In theory, the approximations
made in floating point arithmetic induce a loss of precision, and may
therefore introduce a bias in the probability distribution of the
generated objects.

Denise and Zimmermann \cite{dz99} showed that the complexity of the
floating point implementation is the same as for the RAM model.  They
also show that, if certain precautions are taken (essentially in the
choice of the rounding operator for each operation), a floating point
implementation introduces only a negligible bias in the probability
distribution of the generated objects.  In the case of partial
injections, using the standard rounding operator does not seem
experimentally to produce a significant bias, but it is not
theoretically proved.

\section{On the rank of a size $n$ subgroup}\label{sec:other}
%%%%%%%%%%%%%%%%%%%%%%%%
%\subsection{Average }

We conclude this paper with a few applications of the above results to
the study of the rank distribution of size $n$ subgroups.  The first
one concerns the expected value of this rank, and the others establish
the intuitive results that finite index (resp.  fixed rank $k$)
subgroups are asymptotically negligible.

Recall that the rank of a subgroup $H$ with a size $n$ graphical
representation $\Gamma$ is equal to $|E(\Gamma)| - n + 1$, where
$|E(\Gamma)|$ is the number of edges of the graph $\Gamma$ (see 
Section~\ref{sec: representation}).

%%% rang moyen
\begin{corollary}\label{cor:averagerank}
    The average rank of a size $n$ subgroup of $F(A)$ is
    $(|A|-1)n-|A|\sqrt n + 1$, with standard deviation $o(\sqrt n)$.
\end{corollary}

\begin{proof}
    Let $\Gamma$ an $A$-graph.  For each letter $a$ of the alphabet
    $A$, the number of $a$-labeled edges is the difference between $n$
    and the number of sequences in the functional graph of the partial
    injection determined by the $a$-labeled edges.  In view of
    Lemma~\ref{lm count sequences}, the number of $a$-labeled edges is
    therefore asymptotically equal to $n - \sqrt n$, with standard
    deviation $o(\sqrt n)$, and the announced result follows.
\end{proof}

\begin{corollary}\label{cor: rank k negligible}
    Let $H$ be a size $n$ subgroup of $F(A)$ and let $k \ge 1$ be an
    integer.  Then the probability that $\rank(H) \le k$ is asymptotically
    $o\left(\frac{1}{n}\right)$.
\end{corollary}    

\begin{proof}
    Corollary \ref{cor:averagerank} shows that the mean value of the
    rank of $H$ is $\E(\rank)=(|A|-1)n-|A| \sqrt n + 1$, with variance
    $\sigma^2(\rank) = o(n)$.
    
    If $\rank(H) \le k$, then in particular $|\rank(H) - \E(\rank)|
    \ge \E(\rank) - k$.  It follows, by Chebyshev's inequality (see
    Equation~(\ref{eq:tchebychev}) above), that
    \begin{eqnarray*}
	\mathbb{P}\{\rank \leq k\} &\leq&
	\mathbb{P}\{|\rank-\E(\rank)|\geq \E(\rank)-k\}\\
	&\leq& \frac{o(n)}{\O(n^2)} = o(\frac 1n).
    \end{eqnarray*}
\end{proof}

\section{Finite index subgroups}\label{sec:permutations}

We saw in Section~\ref{sec: representation} that a finitely generated
subgroup has finite index if and only if, in its graphical
representation, every letter labels a permutation, that is, a partial
injection whose domain is the full set of vertices.  Let us say, in
that case, that the corresponding $A$-graph is a \textit{permutation
$A$-graph}.  Based on this observation, we can adapt our approach to
get a linear time (in average) random generation algorithm.

As in the general case (see Section \ref{sec:algoadmissible}), we use
a rejection algorithm: we repeatedly randomly generate a permutation
of size $n$ for each letter $a\in A$, until the resulting graph is
connected and 1-trim.

Random generation of permutations is a classical object of study, and
it can be performed in time $\O(n)$ (in the RAM model, $\O(n\log n)$ 
in the bit-cost model) using the following algorithm (see
\cite{dev86} for example):
\begin{tabbing}
\quad \= \quad \= \quad \= \quad \kill
$\func{RandomPermutation}{n}$ \qquad \\
\> \keyw{for} $i\in\{1,\cdots,n\}$\\
\> \> P[$i$] = $i$ \\
\> \keyw{for} $i$ from $2$ to $n$\\
\> \> $j = 1 + \func{Uniform}{i}$ \qquad\com{$j$ is a random integer in
$\inter[1,i]$}\\
\> \> Swap $P[i]$ and $P[j]$\\
\> \keyw{return} P\\
\end{tabbing}

\vskip -.5cm

\noindent Note that this algorithm does not require manipulating large
integers.

The efficiency of the rejection algorithm depends on the average 
number of rejects, and hence on the probability, for a permutation 
$A$-graph to be connected and 1-trim. Trimness is a moot point since 
a permutation $A$-graph never has any leaf.

Connectedness is not guaranteed, but we note that Dixon \cite{Dixon}
uses Bender's theorem (Theorem~\ref{thbender} above) to compute the
asymptotic expansion of the probability for a pair (or a $r$-tuple) of
size $n$ permutations to generate a transitive subgroup of
$\mathbb{S}_n$, that is, to define a connected permutation $A$-graph.
He shows in particular that this probability is of the form
$1-1/n^{r-1} + \O(1/n^{2(r-1)})$. 

Thus the average number of rejects tends to 0 when $n$ tends to
infinity, and the average case complexity of the random generation of
an admissible permutation $A$-graph is $\O(n)$ (in the RAM model, 
$\O(n\log n)$ in the bitcost model).

We can also show that finite index subgroups are 
asymptotically negligible among subgroups of a given size. 

\begin{proposition}\label{prop: fi negligible}
    The probability for a randomly chosen size $n$ subgroup of $F(A)$
    to have finite index is $\O(n^{r/4} e^{-2r\sqrt n})$.  In
    particular, it is $o\left(n^{-k}\right)$ for any $k\ge 1$.
\end{proposition}    

\begin{proof}
Let $r = \rank(F) = |A|$.  The number of size $n$ finite index
subgroups is at most the number of $r$-tuples of permutations, namely
$n!^r$. The number of size $n$ subgroups is, according to the 
discussion in this paper, equivalent to the number of $r$-tuples of 
partial injections, that is, it is equal to $I_n^r (1 + o(1))$.

Thus the probability that a size $n$ subgroup has finite index is at
most equal to $\left(\frac{n!}{I_n}\right)^r(1+o(1))$.  By
Proposition~\ref{prop cor du col} (applied with $p = 0$), we know that
$I_n/n!  = \O(n^{-1/4} e^{2\sqrt n})$.  Therefore
$$\left(\frac{n!}{I_n}\right)^r = \O\left(n^{r/4} e^{-2r\sqrt n}\right),$$
which converges to $0$ faster than the inverse of any polynomial.
\end{proof}

%%%%%%%%%%%%%%%%%%%%%%%%
\section{A few questions}
%%%%%%%%%%%%%%%%%%%%%%%%

A first question, prompted by Proposition~\ref{prop: fi negligible},
is the following.  Even though finite index subgroups are negligible
among finitely generated subgroups, we saw in
Section~\ref{sec:permutations} how to randomly generate them.  When
$k$ is fixed, rank $k$ subgroups are also asymptotically negligible
among finitely generated subgroups (Corollary~\ref{cor: rank k
negligible}).  Can we find an efficient random generation algorithm
for these subgroups?

\medskip

Our second question is related with another method used in the
literature to generate subgroups (not only for free groups).  This
method is based on the idea of randomly generating a $k$-tuple of
elements that generate the subgroup --- with $k$ fixed and, say, the
maximal length of the generators being allowed to tend to infinity.
It is used for instance by Jutsikawa \cite{jitsukawa} to study the
distribution of malnormal subgroups in free groups.  We refer also to
Martino, Turner and Ventura \cite{MTV} on the distribution of
monomorphisms between free groups, and to Miasnikov and Ushakov
\cite{MiasUsha} for a survey of this technique in relation with
group-based cryptography.

The question that arises in this context is to compare the
distribution of subgroups that occurs with this generation scheme and
the distribution we considered in this paper.  They must be different
since we fix the size of the subgroups generated, whereas they fix the
number and the maximal length of a set of generators, which may lead
to graphical representations of varying size.  One must also take into
consideration the fact that each subgroup is generated by a
potentially large number of $k$-tuples of generators.  However, it
remains possible that generic properties (those that have
asymptotically probability 1) coincide for both distributions.

%%%%%%%%%%%%%%%%%%%%%%%%

%%%%%%%%%%%%%%%%%%%%%%%%
%%%%%%%%%%%%%%%%%%%%%%%%
%%%%%%%%%%%%%%%%%%%%%%%%
\end{document}